\documentclass[12pt,english]{article}
\usepackage{mathptmx}

\usepackage[T1]{fontenc}
\usepackage[latin9]{inputenc}
\usepackage[a4paper]{geometry}
\usepackage{babel}
\usepackage{amsmath}
\usepackage{amsthm}
\usepackage{amssymb}
\usepackage{setspace}
\usepackage[numbers,square,sort,comma,numbers]{natbib}
\usepackage[unicode=true,
 bookmarks=true,bookmarksnumbered=false,bookmarksopen=false,
 breaklinks=false,pdfborder={0 0 1},backref=false,colorlinks=false]
 {hyperref}
\hypersetup{pdftitle={Use Lyx for typesetting: A template for empirical research projects},
 pdfauthor={First Author and Second Author}}

\makeatletter
\newcommand*{\LyxTextAccent}[3][0ex]{%
  \hmode@bgroup\ooalign{\null#3\crcr\hidewidth
  \raise#1\hbox{#2}\hidewidth}\egroup}
\newcommand{\LyxAccentSize}[1][\sf@size]{%
  \check@mathfonts\fontsize#1\z@\math@fontsfalse\selectfont
}
\ProvideTextCommandDefault{\textcommabelow}[1]{
  \LyxTextAccent[-.31ex]{\LyxAccentSize,}{#1}}

\newenvironment{lyxlist}[1]
	{\begin{list}{}
		{\settowidth{\labelwidth}{#1}
		 \setlength{\leftmargin}{\labelwidth}
		 \addtolength{\leftmargin}{\labelsep}
		 }}
	{\end{list}}

\hypersetup{
    colorlinks=true,       
    linkcolor=blue,          
    citecolor=blue,        
    filecolor=blue,      
    urlcolor=blue,           
    breaklinks=true
}

\usepackage{booktabs}
\usepackage{longtable}
\usepackage{pdflscape}
\usepackage[square,sort,comma,numbers]{natbib}
\usepackage{amsthm}
\usepackage[title]{appendix}

\makeatother

\begin{document}
\title{\onehalfspacing{}Non-linear, solvable, residually $p$ groups}
\author{Donsung Lee}
\date{September 23, 2023}

\maketitle
\medskip{}

\begin{abstract}
\begin{spacing}{0.9}
\noindent In 2005, Borisov and Sapir proved that ascending HNN extensions
of finitely generated linear groups are residually finite. Subsequently,
Dru\textcommabelow{t}u and Sapir noted the existence of finitely generated
non-linear residually finite groups based on the work of Borisov and
Sapir. In 2017, Kharlampovich, Myasnikov and Sapir showed that there
exist finitely generated non-linear solvable residually finite groups.
In this paper, we construct the first examples of finitely generated
non-linear solvable residually 2 groups.\vspace{1cm}

\medskip{}

\end{spacing}
\end{abstract}
\noindent $ $\theoremstyle{definition}
\newtheorem{definition}{Definition}[section]
\theoremstyle{remark}
\newtheorem{theorem}{Theorem}[section]
\newtheorem{lemma}[theorem]{Lemma}
\newtheorem{corollary}[theorem]{Corollary}
\newtheorem{remark}[theorem]{Remark}

\newtheorem*{ack}{Acknowledgements}
\newtheorem*{lemma21}{Lemma \textnormal{2.1}}
\newtheorem*{claim1}{Claim \textnormal{1}}
\newtheorem*{proofclaim1}{Proof of Claim \textnormal{1}}
\newtheorem*{claim2}{Claim \textnormal{2}}
\newtheorem*{proofclaim2}{Proof of Claim \textnormal{2}}
\newtheorem*{claim3}{Claim \textnormal{3}}
\newtheorem*{proofclaim3}{Proof of Claim \textnormal{3}}
\newtheorem*{claim4}{Claim \textnormal{4}}
\newtheorem*{proofclaim4}{Proof of Claim \textnormal{4}}
\newtheorem*{claim5}{Claim \textnormal{5}}
\newtheorem*{proofclaim5}{Proof of Claim \textnormal{5}}
\newtheorem*{claim6}{Claim \textnormal{6}}
\newtheorem*{proofclaim6}{Proof of Claim \textnormal{6}}
\newtheorem*{claim7}{Claim \textnormal{7}}
\newtheorem*{proofclaim7}{Proof of Claim \textnormal{7}}
\newtheorem*{claim8}{Claim \textnormal{8}}
\newtheorem*{proofclaim8}{Proof of Claim \textnormal{8}}
\newtheorem*{proofthm12}{Proof of Theorem \textnormal{1.2}}
\newtheorem*{proofthm13}{Proof of Theorem \textnormal{1.3}}
\newtheorem*{proofthm15}{Proof of Theorem \textnormal{1.5}}
\newtheorem*{prooflemma21}{Proof of Lemma \textnormal{2.1}}

\section{Introduction}

For a group property $X$, a group $G$ is \emph{residually X} if
for every nontrivial element $g$, there exists a homomorphism $h:G\to H$
to a group $H$ with property $X$ such that $h\left(g\right)\ne1$.
A group $G$ is \emph{linear} if there exists a field $F$ and an
injective homomorphism $\phi:G\to\mathrm{GL}\left(m,\,F\right)$ for
some integer $m$. A group $G$ is \emph{virtually }$X$ if $G$ has
a subgroup which has property $X$ and is of finite index.

Mal'cev\citep{MR0003420} established in 1940 that a finitely generated
linear group is residually finite. Platonov \citep{MR0231897} further
demonstrated in 1968 that a finitely generated linear group is virtually
residually (finite $p$) for some prime $p$. Henceforth, we will
refer to a group as \emph{residually }$p$ if it is residually (finite
$p$).

In \citep{MR2138070}, Borisov and Sapir demonstrated that ascending
HNN extensions of finitely generated linear groups are residually
finite. Given that finitely generated linear groups are themselves
residually finite, a natural question arises: are there ascending
HNN extensions of linear groups that are non-linear? A result of Wehrfritz
\citep{MR0367080} is classical.

\noindent \begin{theorem}

\noindent \citep[Corollary 2.4]{MR0367080} Let $r$ and $s$ be integers
with $\left|r\right|,\left|s\right|>1$. Then
\begin{align*}
G & =\left\langle a,b,h\;|\;hah^{-1}=a^{r},\;hbh^{-1}=b^{s}\right\rangle 
\end{align*}
is non-linear.

\noindent \end{theorem}

Using this theorem, Dru\textcommabelow{t}u\textendash Sapir \citep{MR2115010}
provided the first example of non-linear residually finite 1-related
group, and offered an alternative proof of Theorem 1.1. Following
the work of Borisov\textendash Sapir \citep{MR2138070} and Dru\textcommabelow{t}u\textendash Sapir
\citep{MR2115010}, more examples of finitely generated non-linear
residually finite groups have been discovered or re-discovered. In
particular, Tholozan and Tsouvalas in \citep{tholozan2022residually}
provided the first examples of finitely generated non-linear hyperbolic
residually finite groups.

In \citep{MR2533795}, Borisov and Sapir also proved that ascending
HNN extensions of finitely generated free groups are virtually residually
$p$ for every sufficiently large prime $p$. Note that this theorem
also implies that there are finitely generated non-linear residually
$p$ groups from Wehrfritz's non-linear groups, since a virtually
linear group is linear by the induced representation. Recently, Chong
and Wise \citep{MR4388367} showed that there exist uncountably many
non-linear finitely generated residually finite groups.

On the other hand, Kharlampovich, Myasnikov and Sapir in \citep{MR3671739}
showed that there are finitely generated non-linear solvable residually
finite groups of solvability class 3, by using Bou-Rabee's depth function.
The non-linear groups of \citep{MR2115010} and \citep{MR3671739}
are mutually exclusive; the former have a nonabelian free subgroup
and the latter are solvable. We construct the first examples of finitely
generated non-linear solvable residually $p$ groups.

Denote by $\mathbb{F}_{2}$ the finite field of order 2. For each
nonzero integer $n$, define an injective endomorphism 
\begin{align*}
\rho_{n} & :\mathrm{GL}\left(3,\,\mathbb{F}_{2}\left[t,\,t^{-1}\right]\right)\to\mathrm{GL}\left(3,\,\mathbb{F}_{2}\left[t,\,t^{-1}\right]\right)
\end{align*}
by
\begin{align*}
\rho_{n}\left(\begin{array}{ccc}
a_{11}\left(t\right) & a_{12}\left(t\right) & a_{13}\left(t\right)\\
a_{21}\left(t\right) & a_{22}\left(t\right) & a_{23}\left(t\right)\\
a_{31}\left(t\right) & a_{32}\left(t\right) & a_{33}\left(t\right)
\end{array}\right) & :=\left(\begin{array}{ccc}
a_{11}\left(t^{n}\right) & a_{12}\left(t^{n}\right) & a_{13}\left(t^{n}\right)\\
a_{21}\left(t^{n}\right) & a_{22}\left(t^{n}\right) & a_{23}\left(t^{n}\right)\\
a_{31}\left(t^{n}\right) & a_{32}\left(t^{n}\right) & a_{33}\left(t^{n}\right)
\end{array}\right).
\end{align*}

Define two matrices $c$ and $d$ in $\mathrm{GL}\left(3,\,\mathbb{F}_{2}\left[t,t^{-1}\right]\right)$
by
\begin{align*}
c & :=\left(\begin{array}{ccc}
t & 1+t & 1+t\\
1+t & t & 1+t\\
1+t & 1+t & t
\end{array}\right),\;d:=\left(\begin{array}{ccc}
t & 0 & 0\\
0 & 1 & 0\\
0 & 0 & 1
\end{array}\right).
\end{align*}

Let $G$ be the group generated by $c$ and $d$. By induction, for
every nonzero integer $n$ we have
\begin{align*}
c^{n} & =\rho_{n}\left(c\right),\;d^{n}=\rho_{n}\left(d\right).
\end{align*}

Thus, the restriction $\rho_{n}|_{G}$ becomes an injective endomorphism
from $G$ into itself. Denote by $H_{n}$ the HNN extension of $G$
by $\rho_{n}|_{G}$. We state the main results.

\noindent \begin{theorem}

For every nonzero integer $n$, $H_{n}$ is solvable of class 4.

\noindent \end{theorem}

\noindent \begin{theorem}

For every integer $n$ such that $\left|n\right|>1$, $H_{n}$ is
non-linear.

\noindent \end{theorem}

In our proof of Theorem 1.3 in Section 3, we employ the method (Theorem
3.1) that Wehrfritz used to establish Theorem 1.1.

Generalizing $\rho_{n}$, define an injective endomorphism
\begin{align*}
\rho_{n}^{m} & :\mathrm{GL}\left(m,\,\mathbb{F}_{2}\left[t,\,t^{-1}\right]\right)\to\mathrm{GL}\left(m,\,\mathbb{F}_{2}\left[t,\,t^{-1}\right]\right)
\end{align*}
for every integer $m$ such that $m\ge3$, as each $A\in\mathrm{GL}\left(m,\,\mathbb{F}_{2}\left[t,t^{-1}\right]\right)$
satisfies
\begin{align*}
\left(\rho_{n}^{m}\left(A\right)\right)_{ij}\left(t\right) & =A_{ij}\left(t^{n}\right).
\end{align*}

\noindent \begin{corollary}

For every pair of integers $\left(m,n\right)$ such that $m\ge3$
and $\left|n\right|>1$, the HNN extension of $\mathrm{GL}\left(m,\,\mathbb{F}_{2}\left[t,t^{-1}\right]\right)$
by $\rho_{n}^{m}$ is finitely generated, non-linear, and residually
finite.

\noindent \end{corollary}
\begin{proof}

Suslin in \citep{MR0472792} proved that for $m\ge3$, $\mathrm{SL}\left(m,\,\mathbb{F}_{2}\left[t,t^{-1}\right]\right)$
is finitely generated by elementary matrices. Thus, $\mathrm{GL}\left(m,\,\mathbb{F}_{2}\left[t,t^{-1}\right]\right)$
is also finitely generated, and by \citep{MR2138070}, the HNN extension
is residually finite. The non-linearity follows from Theorem 1.3 and
the fact that a subgroup of a linear group is also linear.$\qedhere$

\noindent \end{proof}
\begin{theorem}

For every integer $n$, $H_{n}$ is a residually $2$ group if and
only if $n$ is odd.

\noindent \end{theorem}

Combining Theorem 1.2, Theorem 1.3 and Theorem 1.5, we establish the
result stated in the abstract.

\noindent \begin{corollary}

For every integer $n$ such that $\left|n\right|>1$ and $n$ is odd,
$H_{n}$ is finitely generated, non-linear, solvable, and residually
2.

\noindent \end{corollary}

Note that $H_{n}$ is an ascending HNN extension of a finitely generated
linear group, making it residually finite by Borisov\textendash Sapir
\citep{MR2138070}. In Section 4, we will prove that $H_{n}$ is an
extension of the Baumslag\textendash Solitar group $BS\left(1,n\right)$
by a nilpotent group (Lemma 4.5). The group $H_{n}$ shares solvability
and residual properties with $BS\left(1,n\right)$. For any nonzero
integer $n$, both $BS\left(1,n\right)$ and $H_{n}$ are solvable
(Theorem 1.2), residually finite \citep{MR0285589}, and residually
2 if and only if $n$ is odd \citep[Teorema 2]{MR3076943}, (Theorem
1.5).

\noindent \begin{remark}

The argument we have used to prove Corollary 1.4 can be applied to
any supergroup $\widetilde{G}$ of $G$ in $\mathrm{GL}\left(m,\,\mathbb{F}_{2}\left[t,t^{-1}\right]\right)$
such that $\rho_{n}^{m}$ induces an endomorphism on $\widetilde{G}$.
For example, since $\rho_{-1}$ is an involution of $\mathrm{GL}\left(3,\,\mathbb{F}_{2}\left[t,t^{-1}\right]\right)$,
we define the unitary group of the form induced by $\rho_{-1}$:
\begin{align*}
U & :=\left\{ A\in\mathrm{GL}\left(3,\,\mathbb{F}_{2}\left[t,\,t^{-1}\right]\right)\;:\;A^{T}\rho_{-1}\left(A\right)=I\right\} ,
\end{align*}
where $A^{T}$ means the matrix transpose. Then, $U$ is finitely
generated by $c,d$ and the orthogonal subgroup of $\mathrm{SL}\left(3,\,\mathbb{F}_{2}\right)$,
isomorphic to the symmetric group $S_{3}$, which is easily proved
by the method of Cooper and Long in \citep{MR1431138} using the action
on the Bruhat\textendash Tits building. Therefore, for any integer
$n$ such that $\left|n\right|>1$, the HNN extension of $U$ by $\rho_{n}|_{U}$
is finitely generated, non-linear, and residually finite.

\noindent \end{remark}

The rest of this paper is organized as follows. In Section 2, we introduce
basic definitions and prove $\mathrm{Theorem\;1.2}$. We prove $\mathrm{Theorem\;1.3}$
in Section 3, and $\mathrm{Theorem\;1.5}$ in Section 4.

\noindent \begin{ack}

\noindent This paper is supported by National Research Foundation
of Korea (grant number 2020R1C1C1A01006819) and Samsung Science and
Technology Foundation (project number SSTF-BA2001-01). The author
thanks Sang-hyun Kim for valuable and encouraging comments.

\noindent \end{ack}

\section{Proof of Theorem 1.2}

Let us define $b_{0}:=dc^{-1}$ in $G$ and define a matrix $X\in\mathrm{GL}\left(3,\,\mathbb{F}_{2}\left(t\right)\right)$
by
\begin{align*}
X & :=\left(\begin{array}{ccc}
0 & 0 & 1\\
1+t^{-1} & 1+t & 1+t\\
1+t^{-1} & 1+t^{-1} & t^{-1}
\end{array}\right).
\end{align*}

A direct computation yields
\begin{align*}
 & Xb_{0}X^{-1}=\left(\begin{array}{ccc}
0 & 0 & 1\\
1 & 0 & 1\\
0 & 1 & 1
\end{array}\right),\\
 & Xd^{-1}X^{-1}=\left(\begin{array}{ccc}
1 & 0 & 0\\
1 & 1+t^{-1} & 1\\
1 & t^{-1} & 0
\end{array}\right),
\end{align*}
where $\left(\begin{array}{ccc}
0 & 0 & 1\\
1 & 0 & 1\\
0 & 1 & 1
\end{array}\right)$ (resp. $\left(\begin{array}{ccc}
1 & 0 & 0\\
1 & 1+t^{-1} & 1\\
1 & t^{-1} & 0
\end{array}\right)$) is equal to $t^{-1}x$ (resp. $t^{-4}yxy$) in $\mathrm{GL}\left(3,\,\mathbb{F}_{2}\left(t\right)\right)$
as in \citep[Appendix A]{lee2023characterizing}. Thus, we use the
results on $x$ and $yxy$ in $\mathrm{PGL}\left(3,\,\mathbb{F}_{2}\left(t\right)\right)$
in \citep[Appendix A]{lee2023characterizing} to compute the presentation
of $G$, generated by $c$ and $d$, via
\begin{align}
\mathrm{SL}\left(3,\,\mathbb{F}_{2}\left(t\right)\right) & \cong\mathrm{PSL}\left(3,\,\mathbb{F}_{2}\left(t\right)\right).
\end{align}

From now on, for a list $W$ of elements in a group $A$, we denote
by $\left\langle W\right\rangle _{A}$ (resp. $\left\langle \!\!\left\langle W\right\rangle \!\!\right\rangle _{A}$)
the subgroup generated (resp. normally generated) by the elements
in $W$. We omit the subscript $A$ when there is no room for confusion.
We use the commutator convention $\left[a,b\right]=a^{-1}b^{-1}ab$.

Define a subgroup $NC:=\left\langle d^{m}b_{0}d^{-m}\;:\;m\in\mathbb{Z}\right\rangle $
of $G$. Note that $G$ splits over $NC$ satisfying $G=NC\rtimes\left\langle d\right\rangle $
via the determinant map. In addition, for each nonzero integer $i$,
put
\begin{align*}
b_{i} & :=\left[b_{0},\,d^{i}\right].
\end{align*}
\begin{lemma}

The group $NC$ is presented by the generators $\left\{ b_{i}\::\:i\in\mathbb{Z}\right\} $
and the following relations:
\begin{lyxlist}{00.00.0000}
\item [{$\left.1_{b}\right)$}] \noindent $1=b_{i}^{4},\;i\in\mathbb{Z},$
\item [{$\left.2_{b}\right)$}] \noindent $\left[b_{0},b_{i}\right]=b_{i}^{2},\;i\ne0,$
\item [{$\left.3_{b}\right)$}] \noindent $\left[b_{i},b_{j}\right]=b_{i}^{2}b_{j-i}^{2}b_{j}^{2},\;i<j,\;i\ne0\ne j,$
\item [{$\left.4_{b}\right)$}] \noindent $b_{i}^{2}=b_{-i}^{2},\;i>0.$
\end{lyxlist}
\noindent \end{lemma}
\begin{proof}

See Appendix A. $\qedhere$

\noindent \end{proof}

Put $E_{2^{i}}$ as the elementary abelian group of order $2^{i}$.
Given two positive integers $i\le j$, define a map $\mu_{ij}:E_{2^{i}}\to E_{2^{j}}$
carrying the $m$-th standard generator of $E_{2^{i}}$ to the $m$-th
standard generator of $E_{2^{j}}$. Then, the pair $\left\langle E_{2^{i}},\,\mu_{ij}\right\rangle $
is a direct system, and define $E_{2^{\infty}}$ to be the direct
limit. 

\noindent \begin{corollary}

The derived subgroup of $NC$ is isomorphic to $E_{2^{\infty}}$.
Moreover, via the abelianization the group $NC$ admits a following
short exact sequence: 
\begin{align*}
1 & \longrightarrow E_{2^{\infty}}\longrightarrow NC\longrightarrow\mathbb{Z}/4\mathbb{Z}\times E_{2^{\infty}}\longrightarrow1.
\end{align*}
\end{corollary}
\begin{proof}

According to Lemma A.1 (see Appendix A), the center of $NC$ contains
\[
\left\langle b_{i}^{2}\;:\;i\in\mathbb{Z}\right\rangle _{NC}=\left\langle b_{i}^{2}\;:\;i\ge0\right\rangle _{NC}\simeq E_{2^{\infty}}.
\]

Furthermore, due to the relation $\left.2_{b}\right)$, every $b_{i}^{2}$
for $i\ge1$ is represented by a commutator $\left[b_{0},b_{i}\right]$.
Thus, the derived subgroup of $NC$ contains a normal subgroup
\[
\left\langle b_{i}^{2}\;:\;i\ge1\right\rangle _{NC}\simeq E_{2^{\infty}}.
\]

Given the presentation of $NC$ provided by Lemma 2.1, the quotient
group $NC/\left\langle b_{i}^{2}\;:\;i\ge1\right\rangle _{NC}$ is
generated by the image of $b_{i}$ for $i\in\mathbb{Z}$, and is abelian,
as indicated by $\left.2_{b}\right)$ and $\left.3_{b}\right)$ in
$\mathrm{Lemma}\;2.1$. Consequently, $\left\langle b_{i}^{2}\;:\;i\ge1\right\rangle _{NC}$
also contains the derived subgroup. $\qedhere$

\noindent \end{proof}
\begin{corollary}

The group $G$ is solvable of class 3.

\noindent \end{corollary}
\begin{proof}

Combining $G=NC\rtimes\left\langle d\right\rangle $ with the short
exact sequence on $NC$ in Corollary 2.2, we conclude that $G$ is
solvable of class 3. $\qedhere$

\noindent \end{proof}

Let us consider $H_{n}$. Put $t$ as the additional generator of
$H_{n}$ associated with the HNN extension of $G$. Define subgroups
$\widetilde{NC}_{n}$ and $\widetilde{G}_{n}$ of $H_{n}$ by
\begin{align*}
\widetilde{NC}_{n} & :=\bigcup_{i=0}^{\infty}t^{-i}NCt^{i},\;\widetilde{G}_{n}:=\bigcup_{i=0}^{\infty}t^{-i}Gt^{i}.
\end{align*}

Observe that $\det\left(c\right)=t=\det\left(d\right)$ by direct
computation. Because $NC=\ker\left(\det|_{G}\right)$ and a word $g\left(c,d\right)$
in $c,d$ represents an element of $G$, $\rho_{n}$ induces an endomorphism
on $NC$. Therefore, whenever $i_{1}<i_{2}$ we have
\begin{align*}
\left(t^{-i_{1}}NCt^{i_{1}}\right) & \subset\left(t^{-i_{2}}NCt^{i_{2}}\right),\;\left(t^{-i_{1}}Gt^{i_{1}}\right)\subset\left(t^{-i_{2}}Gt^{i_{2}}\right).
\end{align*}
\begin{lemma}

The group $\widetilde{NC}_{n}$ is nilpotent of class 2 for every
nonzero integer $n$.

\noindent \end{lemma}
\begin{proof}

By the presentation in Lemma 2.1, the quotient group $NC/\left\langle b_{i}^{2}\;:\;i\ge0\right\rangle _{NC}$
is abelian and isomorphic to $E_{2^{\infty}}$. Therefore, for any
element $g\in NC$, $\left\langle b_{i}^{2}\;:\;i\ge0\right\rangle _{NC}$
includes $g^{2}$.

According to Lemma 2.1, $\widetilde{NC}_{n}$ is generated by $\left\{ t^{-i}b_{j}t^{i}\::\;j\in\mathbb{Z},\:i\ge0\right\} $.
Let us define
\begin{align*}
a_{i,j} & :=t^{-i}b_{j}t^{i}.
\end{align*}

Since $t^{-i}NCt^{i}$ is isomorphic to $NC$, the center of $\widetilde{NC}_{n}$
contains
\begin{align*}
\left\langle a_{i,j}^{2}\;:\;j\ge0,\:i\ge0\right\rangle _{\widetilde{NC}_{n}}.
\end{align*}

As in the proof of Corollary 2.2, the derived subgroup of $\widetilde{NC}_{n}$
contains $\left\langle a_{i,j}^{2}\;:\;j\ge1,\:i\ge0\right\rangle _{\widetilde{NC}_{n}}$,
since $a_{i,j}^{2}=\left[a_{i,0},\,a_{i,j}\right]$ from the relation
$\left.2_{b}\right)$ in Lemma 2.1. Define $\widetilde{Q}_{n}$ by
\begin{align*}
\widetilde{Q}_{n} & :=\widetilde{NC}_{n}/\left\langle a_{i,j}^{2}\;:\;j\ge1,\:i\ge0\right\rangle _{\widetilde{NC}_{n}},
\end{align*}
and define $q_{n}:\widetilde{NC}_{n}\to\widetilde{Q}_{n}$ to be the
quotient map.

Choose two images of generators $q_{n}\left(a_{i_{1},j_{1}}\right)$
and $q_{n}\left(a_{i_{2},j_{2}}\right)$ in $\widetilde{Q}_{n}$.
Without loss of generality, suppose $i_{2}\le i_{1}$. Since $a_{i_{2},j_{2}}\in t^{-i_{1}}NCt^{i_{1}}$
and $t^{-i_{1}}NCt^{i_{1}}\cong NC$, from the short exact sequence
of Corollary 2.2, the images of $a_{i_{1},j_{1}}$ and $a_{i_{2},j_{2}}$
commute in the quotient group
\begin{align*}
t^{-i_{1}}NCt^{i_{1}}/\left\langle a_{i_{1},j}^{2}\;:\;j\ge1\right\rangle _{t^{-i_{1}}NCt^{i_{1}}} & =t^{-i_{1}}NCt^{i_{1}}/\left\langle a_{i,j}^{2}\;:\;j\ge1,\:i_{1}\ge i\ge0\right\rangle _{t^{-i_{1}}NCt^{i_{1}}}.
\end{align*}

In other words, the subgroup $\left\langle a_{i_{1},j}^{2}\;:\;j\ge1\right\rangle _{\widetilde{NC}_{n}}$
includes the commutator $\left[a_{i_{1},j_{1}},a_{i_{2},j_{2}}\right]$.
Thus, $q_{n}\left(a_{i_{1},j_{1}}\right)$ and $q_{n}\left(a_{i_{2},j_{2}}\right)$
commute in $\widetilde{Q}_{n}$, which implies that $\widetilde{Q}_{n}$
is abelian. We conclude that $\left\langle a_{i,j}^{2}\;:\;j\ge1,\:i\ge0\right\rangle _{\widetilde{NC}_{n}}$
is the derived subgroup of $\widetilde{NC}_{n}$. Because the center
contains the derived subgroup, $\widetilde{NC}_{n}$ is nilpotent
of class 2. $\qedhere$

\noindent \end{proof}
\begin{lemma}

The group $\widetilde{G}_{n}$ has a faithful linear representation
into $\mathrm{GL}\left(3,\,\mathcal{P}\left(\mathbb{F}_{2}\right)\right)$,
where $\mathcal{P}\left(\mathbb{F}_{2}\right)$ is the field of Puiseux
series with coefficients in $\mathbb{F}_{2}$. Moreover, via the determinant
map, $\widetilde{G}_{n}$ admits the following split short exact sequence:
\begin{align*}
1 & \longrightarrow\widetilde{NC}_{n}\longrightarrow\widetilde{G}_{n}\longrightarrow\mathbb{Z}\left[\frac{1}{n}\right]\longrightarrow1,
\end{align*}
where $\mathbb{Z}\left[\frac{1}{n}\right]$ is the additive group
of the ring of $S$-integers.

\noindent \end{lemma}
\begin{proof}

Suppose the indeterminate of $\mathcal{P}\left(\mathbb{F}_{2}\right)$
is $t$. Denote by $\mathbb{F}_{2}\left[\mathcal{P}\right]$ the subring
of $\mathcal{P}\left(\mathbb{F}_{2}\right)$ of series of finite length.
For any coprime pair of integers $\left(a,b\right)$ such that $a\ne0\ne b$,
abusing notation, define $\rho_{\frac{a}{b}}:\mathrm{GL}\left(3,\,\mathbb{F}_{2}\left[\mathcal{P}\right]\right)\to\mathrm{GL}\left(3,\,\mathbb{F}_{2}\left[\mathcal{P}\right]\right)$
by

\noindent 
\begin{align*}
\rho_{\frac{a}{b}}\left(\begin{array}{ccc}
a_{11}\left(t\right) & a_{12}\left(t\right) & a_{13}\left(t\right)\\
a_{21}\left(t\right) & a_{22}\left(t\right) & a_{23}\left(t\right)\\
a_{31}\left(t\right) & a_{32}\left(t\right) & a_{33}\left(t\right)
\end{array}\right) & :=\left(\begin{array}{ccc}
a_{11}\left(t^{\frac{a}{b}}\right) & a_{12}\left(t^{\frac{a}{b}}\right) & a_{13}\left(t^{\frac{a}{b}}\right)\\
a_{21}\left(t^{\frac{a}{b}}\right) & a_{22}\left(t^{\frac{a}{b}}\right) & a_{23}\left(t^{\frac{a}{b}}\right)\\
a_{31}\left(t^{\frac{a}{b}}\right) & a_{32}\left(t^{\frac{a}{b}}\right) & a_{33}\left(t^{\frac{a}{b}}\right)
\end{array}\right).
\end{align*}

Then, $\rho_{\frac{a}{b}}$ is an automorphism of $\mathrm{GL}\left(3,\,\mathbb{F}_{2}\left[\mathcal{P}\right]\right)$.
For each integer $i\ge0$, define a representation $\mathrm{rep}_{i}:t^{-i}Gt^{i}\to\mathrm{GL}\left(3,\,\mathbb{F}_{2}\left[\mathcal{P}\right]\right)$
as for each $g\in G$,
\begin{align*}
\mathrm{rep}_{i}\left(t^{-i}gt^{i}\right) & :=\rho_{n^{-i}}\left(g\right).
\end{align*}

We see each $\mathrm{rep}_{i}$ is faithful by construction. Moreover,
for any pair of integers $\left(i,j\right)$ such that $i\le j$,
$\mathrm{rep}_{i}$ and $\mathrm{rep}_{j}$ have the same image of
each element in $t^{-i}Gt^{i}$. Indeed, for each $g\in G,$
\begin{align*}
\mathrm{rep}_{i}\left(t^{-i}gt^{i}\right) & =\rho_{n^{-i}}\left(g\right)=\rho_{n^{-j}}\left(\rho_{n^{j-i}}\left(g\right)\right)=\rho_{n^{-j}}\left(t^{j-i}gt^{i-j}\right)=\mathrm{rep}_{j}\left(t^{-i}gt^{i}\right).
\end{align*}

Define a representation $\mathrm{rep}:\widetilde{G}_{n}\to\mathrm{GL}\left(3,\,\mathbb{F}_{2}\left[\mathcal{P}\right]\right)$
as follows: for each $g\in G$ and $i\ge0$,
\begin{align*}
\mathrm{rep}\left(t^{-i}gt^{i}\right) & :=\mathrm{rep}_{i}\left(t^{-i}gt^{i}\right).
\end{align*}

Then, it is faithful, as every $\mathrm{rep}_{i}$ is faithful. By
construction, we have
\begin{align*}
\det\left(\mathrm{rep}\left(\widetilde{G}_{n}\right)\right) & =\left\{ t^{z}\;:\;z\in\mathbb{Z}\left[\frac{1}{n}\right]\right\} .
\end{align*}

Since $\ker\left(\det|_{G}\right)=NC$, we also have
\begin{align*}
\ker\left(\det\circ\mathrm{rep}_{i}\right) & =t^{-i}NCt^{i},
\end{align*}
which implies
\begin{align*}
\ker\left(\det\circ\mathrm{rep}\right) & =\widetilde{NC}_{n},
\end{align*}
and we obtain the short exact sequence desired. Finally, the sequence
splits, as for each nonnegative integer $i$,

\[
\mathrm{rep}\left(t^{-i}dt^{i}\right)=\left(\begin{array}{ccc}
t^{n^{-i}}\\
 & 1\\
 &  & 1
\end{array}\right).\;\qedhere
\]
\end{proof}

\noindent \begin{proofthm12}

By definition, we have
\begin{align*}
\widetilde{G}_{n} & =\left\langle t^{-i}dt^{i},\,t^{-i}b_{0}t^{i}\;:\;i\ge0\right\rangle _{H_{n}}=\left\langle \!\!\left\langle c,\,d\right\rangle \!\!\right\rangle _{H_{n}},
\end{align*}
which implies $H_{n}=\widetilde{G}_{n}\rtimes\left\langle t\right\rangle _{H_{n}}$.
Therefore, from Lemma 2.4 and Lemma 2.5, $H_{n}$ has a subnormal
series:
\begin{align*}
1 & \trianglelefteq\left(\widetilde{NC}_{n}\right)'\trianglelefteq\widetilde{NC}_{n}\trianglelefteq\widetilde{G}_{n}\trianglelefteq H_{n},
\end{align*}
where $\left(\widetilde{NC}_{n}\right)'$ is the derived subgroup
of $\widetilde{NC}_{n}$. Lemma 2.4, and Lemma 2.5 also guarantee
that each factor is abelian. $\qed$

\noindent \end{proofthm12}

\section{Proof of Theorem 1.3}

We start this section with an implicit result in the proof of Theorem
1.1 by Wehrfritz \citep[Corollary 2.4]{MR0367080}.

\noindent \begin{theorem}

\noindent Suppose a group $H$ is generated by two non-torsion elements
$a$ and $b$, and there exists an isomorphism $\iota:H\to H'$ such
that $H'\subset H$ and for an integer $\left|r\right|>1$,
\begin{align*}
\iota\left(a\right) & =a^{r},\;\iota\left(b\right)=b^{r}.
\end{align*}

Then, if the HNN extension $H*_{\iota}$ is linear, there exists a
nonzero integer $s$ such that $\left\langle a^{s},b^{s}\right\rangle $
is nilpotent.

\noindent \end{theorem}
\begin{proof}

We reproduce a corresponding portion of the proof of \citep[Corollary 2.4]{MR0367080}.
Put $t$ as the additional generator of $H*_{\iota}$ associated with
the HNN extension of $H$. Suppose that $H*_{\iota}$ is linear over
a field $F$, and let us identify $H*_{\iota}$ with the image of
the representation. We observe that the subgroups $\left\langle a,t\right\rangle $
and $\left\langle b,t\right\rangle $ are isomorphic to the Baumslag\textendash Solitar
group $BS\left(1,r\right)$, which is solvable.

According to a theorem of Mal'cev \citep[Theorem 3.6]{MR0335656},
there exists an integer $m>0$ such that every solvable linear subgroup
of $H*_{\iota}$ has a triangularizable normal subgroup of a finite
index dividing $m$. Therefore, two elements of $H*_{\iota}$,
\begin{align*}
\left[a^{m},\,t^{-m}\right] & =a^{m\left(r^{m}-1\right)},\;\left[b^{m},\,t^{-m}\right]=b^{m\left(r^{m}-1\right)},
\end{align*}
are unipotent. Choose $s=m\left(r^{m}-1\right)$. By the assumption,
$a^{s}$ is not a torsion element. Moreover, as $a^{s}$ is also unipotent,
the field $F$ must have characteristic 0. By applying \citep[Lemma 2.3]{MR0367080},
we conclude that the subgroup $\left\langle a^{s},b^{s}\right\rangle $
of $H*_{\iota}$ is nilpotent. $\qedhere$

\noindent \end{proof}
\begin{lemma}The group $G$ is presented by the generators $c,d$ and the following
relations:
\begin{lyxlist}{00.00.0000}
\item [{$\left.1_{x}\right)$}] \noindent $1=x_{i}^{4},$ for each integer
$i\ge0,$
\item [{$\left.2_{x}\right)$}] \noindent $1=\left[x_{i}^{2},\,d^{-1}\right],$
for each integer $i\ge0,$
\end{lyxlist}
\noindent where $x_{i}$ is a word in $c$ and $d$, recursively defined
by
\begin{align*}
x_{0} & :=b_{0}=dc^{-1},\;x_{i+1}:=\left[d^{-1},\,x_{i}\right],\;i\ge0.
\end{align*}
\end{lemma}
\begin{proof}

At first, we construct a presentation of $G$ directly from the short
exact sequence:

\noindent 
\begin{align*}
1 & \longrightarrow NC\longrightarrow G\longrightarrow\left\langle d\right\rangle \longrightarrow1,
\end{align*}
and the presentation of $NC$ in Lemma 2.1. Since

\[
b_{-i}=\left[b_{0},\,d^{-i}\right]=d^{i}b_{i}^{-1}d^{-i},
\]
all of $\left.4_{b}\right)$ and the relations involving $b_{i}$,
$i<0$ are redundant. Thus, we have a presentation in \citep[Theorem A.5]{lee2023characterizing}
for $G$, without $\left.1_{o}\right)$ and replacing $yxy$ (resp.
$x$) with $d^{-1}$ (resp. $b_{0}$), with the following relations.
\begin{lyxlist}{00.00.0000}
\item [{$\left.2_{o}\right)$}] \noindent $1=\left[b_{0}^{2},d\right],$
\item [{$\left.1_{b}'\right)$}] \noindent $1=b_{i}^{4},$ for each integer
$i\ge0,$
\item [{$\left.2_{b}'\right)$}] \noindent $\left[b_{0},b_{i}\right]=b_{i}^{2},$
for each integer $i\ge1,$
\item [{$\left.3_{b}'\right)$}] \noindent $\left[b_{i},b_{j}\right]=b_{i}^{2}b_{j-i}^{2}b_{j}^{2},\;1\le i<j.$
\end{lyxlist}
\noindent $\quad\;\:$The presentation we need to prove is also a
simplification of that of \citep[Lemma 2.2]{lee2023characterizing},
whose proof in \citep[Appendix A]{lee2023characterizing} from \citep[Theorem A.5]{lee2023characterizing}
makes no use of the relation $\left.1_{o}\right)$. $\qedhere$

\noindent \end{proof}
\begin{lemma}

There exists an injective endomorphism $\eta:G\to G$ satisfying

\noindent 
\begin{align*}
\eta\left(b_{0}\right) & =\left[d^{-1},\,b_{0}\right],\;\eta\left(d\right)=d.
\end{align*}
\end{lemma}
\begin{proof}

As in the proof of \citep[Lemma 2.3]{lee2023characterizing}, define
a matrix $Y\in\mathrm{GL}\left(3,\,\mathbb{F}_{2}\left(t\right)\right)$
by

\[
Y:=\left(\begin{array}{ccc}
1 & 0 & 0\\
0 & 1 & t^{-1}\\
0 & t\left(1+t\right)^{-1} & t^{-1}\left(1+t\right)^{-1}
\end{array}\right).
\]

Then, a direct computation yields
\begin{align*}
 & Yb_{0}Y^{-1}=\left[d^{-1},\,b_{0}\right],\\
 & YdY^{-1}=d.\qedhere
\end{align*}
\end{proof}
\begin{lemma}

The group $G$ is not finitely presented.

\noindent \end{lemma}
\begin{proof}

Suppose $G$ is finitely presented. Then, there exists a nonnegative
integer $M$ such that $G$ is presented by the generators $c,d$
and the following relations:
\begin{lyxlist}{00.00.0000}
\item [{$\left.1_{x,\,M}\right)$}] \noindent $1=x_{i}^{4},$ for each
integer $i$ such that $M\ge i\ge0,$
\item [{$\left.2_{x,\,M}\right)$}] \noindent $1=\left[x_{i}^{2},\,d^{-1}\right],$
for each integer $i$ such that $M\ge i\ge0.$
\end{lyxlist}
$\quad\;\:$Choose $M_{0}$ as the smallest one among the set of such
integers, and delete every relation involving the index larger than
$M_{0}$. By the presentation above, $\left\langle \!\!\left\langle d\right\rangle \!\!\right\rangle _{G}$
is of index $4$ in $G$. Furthermore, $G$ splits over it satisfying
\begin{align*}
G & =\left\langle \!\!\left\langle d\right\rangle \!\!\right\rangle _{G}\rtimes\left\langle x_{0}\right\rangle _{G}.
\end{align*}

By the Reidemeister\textendash Schreier process, $\left\langle \!\!\left\langle d\right\rangle \!\!\right\rangle _{G}$
is presented by the generators
\begin{align*}
d,\,x_{0}dx_{0}^{-1},\,x_{0}^{2}dx_{0}^{-2},\,x_{0}^{-1}dx_{0},
\end{align*}
and the relators consisting of rewritten words of $x_{0}^{i}Rx_{0}^{-i}$
in the new generators, $i=0,1,2,3$ where $R$ is a relator of $G$.
However, two generators are redundant, because from $\left.2_{x,M}\right)$
we have $x_{0}^{2}dx_{0}^{-2}=d$ and $x_{0}dx_{0}^{-1}=x_{0}^{-1}dx_{0}$.
We have $x_{0}^{-1}d^{-1}x_{0}=d^{-1}x_{1}$ by definition, and
\begin{align*}
\left\langle \!\!\left\langle d\right\rangle \!\!\right\rangle _{G} & =\left\langle x_{1},\,d\right\rangle _{G}.
\end{align*}

On the other hand, the new relators are computed as
\begin{lyxlist}{00.00.0000}
\item [{$\left.1_{y,\,M_{0}-1}\right)$}] \noindent $1=y_{i}^{4},$ for
each integer $i$ such that $M_{0}>i\ge0,$
\item [{$\left.2_{y,\,M_{0}-1}\right)$}] \noindent $1=\left[y_{i}^{2},\,d^{-1}\right],$
for each integer $i$ such that $M_{0}>i\ge0,$
\end{lyxlist}
\noindent where $y_{i}$ is a word in $x_{1}$ and $d$, recursively
defined as
\begin{align*}
y_{0} & :=x_{1},\;y_{i+1}:=\left[d^{-1},\,y_{i}\right],\;i\ge0.
\end{align*}

If $M_{0}>0$, by Lemma 3.3 we have $G=\eta^{-1}\left(\left\langle x_{1},\,d\right\rangle _{G}\right)$,
which contradicts the minimality of $M_{0}$. Therefore, we must assume
$M_{0}=0$. In this case, $x_{1}$ and $d$ freely generate $\left\langle \!\!\left\langle d\right\rangle \!\!\right\rangle _{G}=\left\langle x_{1},\,d\right\rangle _{G}$,
which contradicts the solvability of $G$ established in Corollary
2.3. $\qedhere$

\noindent \end{proof}
\begin{proofthm13}

Take an integer $n$ such that $\left|n\right|>1$, and suppose $H_{n}$
is linear. Then, according to Theorem 3.1, there exists a nonzero
integer $s$ such that $\left\langle c^{s},d^{s}\right\rangle $ is
nilpotent. Since $\rho_{s}$ is injective, $G=\left\langle c,d\right\rangle $
is also nilpotent. However, this leads to a contradiction with $\mathrm{Lemma\;3.4}$,
as a finitely generated nilpotent group is finitely presented. $\qed$

\noindent \end{proofthm13}

\section{Proof of Theorem 1.5}

For a positive integer $n$, let us define a group $Q_{n}$ as
\begin{align*}
Q_{n} & :=G/\left\langle \!\!\left\langle d^{n}\right\rangle \!\!\right\rangle _{G}.
\end{align*}
\begin{lemma}

The group $Q_{n}$ is a semidirect product satisfying

\[
Q_{n}=\left\langle b_{0},\,b_{1},\,\cdots,\,b_{n-1}\right\rangle _{Q_{n}}\rtimes\left\langle d\right\rangle _{Q_{n}},
\]
where $\left\langle d\right\rangle _{Q_{n}}$ is isomorphic to the
cyclic group of order $n$ and $\left\langle b_{0},b_{1},\cdots,b_{n-1}\right\rangle _{Q_{n}}$
is presented by the relations:
\begin{lyxlist}{00.00.0000}
\item [{$\left.1_{b,\,n}\right)$}] \noindent $1=b_{i}^{4},\;0\le i\le n-1,$
\item [{$\left.2_{b,\,n}\right)$}] \noindent $\left[b_{0},b_{i}\right]=b_{i}^{2},\;0<i\le n-1,$
\item [{$\left.3_{b,\,n}\right)$}] \noindent $\left[b_{i},b_{j}\right]=b_{i}^{2}b_{j-i}^{2}b_{j}^{2},\;0<i<j\le n-1,$
\item [{$\left.c_{n}\right)$}] \noindent $b_{i}^{2}=b_{n-i}^{2},\;1\le i\le\left\lfloor \frac{n}{2}\right\rfloor ,$
\end{lyxlist}
where $\left\lfloor \frac{n}{2}\right\rfloor $ is the greatest integer
less than or equal to $\frac{n}{2}$.

\noindent \end{lemma}
\begin{proof}

By the proof of Lemma 3.2, the group $Q_{n}$ is presented by the
generators $b_{0},\,d$ and the following relations:
\begin{lyxlist}{00.00.0000}
\item [{$\left.d_{n}\right)$}] \noindent $1=d^{n},$
\item [{$\left.2_{o}\right)$}] \noindent $1=\left[b_{0}^{2},d\right],$
\item [{$\left.1_{b}'\right)$}] \noindent $1=b_{i}^{4},$ for each integer
$i\ge0,$
\item [{$\left.2_{b}'\right)$}] \noindent $\left[b_{0},b_{i}\right]=b_{i}^{2},$
for each integer $i\ge1,$
\item [{$\left.3_{b}'\right)$}] \noindent $\left[b_{i},b_{j}\right]=b_{i}^{2}b_{j-i}^{2}b_{j}^{2},\;0<i<j,$
\end{lyxlist}
where $b_{i}=\left[b_{0},\;d^{i}\right].$

By $\left.d_{n}\right)$ and the definition of $b_{i}$, for every
integer $i$ such that $i\ne0,-n$, we have $b_{i+n}=b_{i}$ in $Q_{n}$.
Therefore, we transform these relations to
\begin{lyxlist}{00.00.0000}
\item [{$\left.d_{n}\right)$}] \noindent $1=d^{n},$
\item [{$\left.2_{o}\right)$}] \noindent $1=\left[b_{0}^{2},d\right],$
\item [{$\left.1_{b,\,n}\right)$}] \noindent $1=b_{i}^{4},$ for each
integer $i$ such that $0\le i\le n-1,$
\item [{$\left.2_{b,\,n}\right)$}] \noindent $\left[b_{0},b_{i}\right]=b_{i}^{2},$
for each integer $0<i\le n-1,$
\item [{$\left.3_{b}'\right)$}] \noindent $\left[b_{i},b_{j}\right]=b_{i}^{2}b_{j-i}^{2}b_{j}^{2},\;0<i<j,$
\item [{$\left.4_{b}\right)$}] \noindent $b_{i}^{2}=b_{-i}^{2},\;i>0,$
\end{lyxlist}
where $\left.4_{b}\right)$ from Lemma 2.1 is a relation of $NC$,
so we may add this relation for $Q_{n}$. By applying $b_{i+n}=b_{i}$,
$\left.4_{b}\right)$ is reduced to
\begin{align*}
b_{i}^{2} & =b_{n-i}^{2},\;1\le i\le\left\lfloor \frac{n}{2}\right\rfloor ,
\end{align*}
which is the relation $\left.c_{n}\right)$.

\noindent \begin{claim1}

\noindent The relation $\left.3_{b}'\right)$ follows from $\left.3_{b,\,n}\right)$
and $\left.d_{n}\right)$.

\noindent \end{claim1}
\begin{proofclaim1}

Given two pairs of integers $\left(i_{1},j_{2}\right)$, $\left(i_{1},j_{2}\right)$
such that
\begin{align*}
0 & <i_{1}<j_{1},\:0<i_{2}<j_{2},\:i_{1}\equiv i_{2}\:(\mathrm{mod}\:n),\:j_{1}\equiv j_{2}\:(\mathrm{mod}\:n),
\end{align*}
it suffices to prove
\begin{align}
 & \,\left[b_{i_{1}},\,b_{j_{1}}\right]=\left[b_{i_{2}},\,b_{j_{2}}\right],\\
 & b_{i_{1}}^{2}b_{j_{1}-i_{1}}^{2}b_{j_{1}}^{2}\,=b_{i_{2}}^{2}b_{j_{2}-i_{2}}^{2}b_{j_{2}}^{2}.
\end{align}

The equality of (2) directly follows from the equations $b_{i_{1}}=b_{i_{2}}$
and $b_{j_{1}}=b_{j_{2}}$ from $\left.d_{n}\right)$. The eqality
of (3) follows from $b_{i_{1}}=b_{i_{2}}$, $b_{j_{1}}=b_{j_{2}}$,
and $b_{j_{1}-i_{1}}=b_{j_{2}-i_{2}}$, where the last follows from
the congruence
\begin{align*}
j_{1}-i_{1} & \equiv j_{2}-i_{2}\:(\mathrm{mod}\:n).\;\qed
\end{align*}
\end{proofclaim1}

We return to the proof of Lemma 4.1. By Claim 1, the group $Q_{n}$
is presented by the generators $b_{0},d$ and the relations $\left.d_{n}\right)$,
$\left.2_{o}\right)$, $\left.1_{b,\,n}\right)$, $\left.2_{b,\,n}\right)$,
$\left.3_{b,\,n}\right)$, and $\left.c_{n}\right)$. By the presentation,
$Q_{n}$ is a semidirect product $\left\langle \!\!\left\langle b_{0}\right\rangle \!\!\right\rangle _{Q_{n}}\rtimes\left\langle d\right\rangle _{Q_{n}}$,
where $\left\langle d\right\rangle _{Q_{n}}$ is isomorphic to the
cyclic group of order $n$. By the Reidemeister\textendash Schreier
process, we compute the presentation of $\left\langle \!\!\left\langle b_{0}\right\rangle \!\!\right\rangle _{Q_{n}}=\left\langle b_{0},b_{1},\cdots,b_{n-1}\right\rangle _{Q_{n}}$
as required. $\qedhere$ 

\noindent \end{proof}
\begin{corollary}

The center of $\left\langle b_{0},b_{1},\cdots,b_{n-1}\right\rangle _{Q_{n}}$
is generated by $b_{0}^{2},b_{1}^{2},\cdots,b_{\left\lfloor \frac{n}{2}\right\rfloor }^{2}$
and the group $\left\langle b_{0},b_{1},\cdots,b_{n-1}\right\rangle _{Q_{n}}$admits
a short exact sequence on the center
\begin{align*}
1 & \longrightarrow E_{2^{\left\lfloor \frac{n}{2}\right\rfloor +1}}\longrightarrow\left\langle b_{0},\,b_{1},\,\cdots,\,b_{n-1}\right\rangle _{Q_{n}}\longrightarrow E_{2^{n}}\longrightarrow1.
\end{align*}
\end{corollary}
\begin{proof}

It is direct from the presentation of Lemma 4.1. $\qedhere$

\noindent \end{proof}
\begin{corollary}

For each positive integer $n$, there is an injective homomorphism
\begin{align*}
\xi_{n} & :\left\langle b_{0},\,b_{1},\,\cdots,\,b_{n-1}\right\rangle _{G}\to\left\langle b_{0},\,b_{1},\,\cdots,\,b_{2n-1}\right\rangle _{Q_{2n}}
\end{align*}
satisfying
\begin{align*}
\xi_{n}\left(b_{i}\right) & =b_{i},\;0\le i\le n-1.
\end{align*}
\end{corollary}
\begin{proof}

By \citep[Lemma A.3]{lee2023characterizing}, the center of $\left\langle b_{0},b_{1},\cdots,b_{n-1}\right\rangle _{G}$
is $\left\langle b_{0}^{2},b_{1}^{2},\cdots,b_{n-1}^{2}\right\rangle _{G}$,
and induces a short exact sequence:
\begin{align*}
1 & \longrightarrow E_{2^{n}}\longrightarrow\left\langle b_{0},\,b_{1},\,\cdots,\,b_{n-1}\right\rangle _{G}\longrightarrow E_{2^{n}}\longrightarrow1.
\end{align*}

On the other hand, the center of $\left\langle b_{0},b_{1},\cdots,b_{2n-1}\right\rangle _{Q_{2n}}$
is $\left\langle b_{0}^{2},b_{1}^{2},\cdots,b_{n}^{2}\right\rangle _{Q_{2n}}$,
and induces a short exact sequence:
\begin{align*}
1 & \longrightarrow E_{2^{n+1}}\longrightarrow\left\langle b_{0},\,b_{1},\,\cdots,\,b_{2n-1}\right\rangle _{Q_{2n}}\longrightarrow E_{2^{2n}}\longrightarrow1.
\end{align*}

By applying the four lemma, we obtain the injectivity. $\qedhere$

\noindent \end{proof}

For each pair of integers $\left(m,n\right)$ such that $n$ is odd
and $m>0$, define a group $P_{n,\,m}$ to be
\[
P_{n,\,m}:=H_{n}/\left\langle \!\!\left\langle d^{2^{m}},\,t^{2^{m}}\right\rangle \!\!\right\rangle _{H_{n}},
\]
and let $\pi_{n,\,m}:H_{n}\to P_{n,\,m}$ be the quotient map.

\noindent \begin{lemma}

For each pair of integers $\left(m,n\right)$ such that $n$ is odd
and $m>0$, the group $P_{n,\,m}$ is a semidirect product

\[
P_{n,\,m}=Q_{2^{m}}\rtimes\left\langle t\right\rangle _{P_{n,\,m}},
\]
where $\left\langle t\right\rangle _{P_{n,\,m}}$ is isomorphic to
the cyclic group of order $2^{m}$ and $Q_{2^{m}}$ is generated by
$c,d$ in $P_{n,\,m}$.

\noindent \end{lemma}
\begin{proof}

From the presentation of $G$ in Lemma 3.2, $P_{n,\,m}$ is presented
by the generators $c,d,t$, where $c$ (resp. $d$) is the image of
$c$ (resp. $d$) in $H_{n}$, and the relators
\begin{align*}
R_{G},\,tct^{-1}c^{-n},\,tdt^{-1}d^{-n},\,d^{2^{m}},\,t^{2^{m}},
\end{align*}
where $R_{G}$ is the collection of the relators of $G$. From this
presentation, we have
\begin{align*}
P_{n,\,m}/\left\langle \!\!\left\langle c,\,d\right\rangle \!\!\right\rangle _{P_{n,\,m}} & \simeq\left\langle t\right\rangle _{P_{n,\,m}}\simeq\mathbb{Z}/2^{m}\mathbb{Z}.
\end{align*}

Therefore, $P_{n,\,m}$ is a semidirect product $\left\langle \!\!\left\langle c,d\right\rangle \!\!\right\rangle _{P_{n,\,m}}\rtimes\left\langle t\right\rangle _{P_{n,\,m}}$.
By the Reidemeister\textendash Schreier process, we compute the presentation
of $\left\langle \!\!\left\langle c,d\right\rangle \!\!\right\rangle _{P_{n,\,m}}$.
A new generating set is

\[
\left\{ t^{i}ct^{-i},\,t^{i}dt^{-i}\;:\;0\le i\le2^{m}-1\right\} ,
\]
which is again generated by $c$ and $d$ by the HNN construction.
Likewise, any $t^{i}$ conjugation of a relator from $R_{G}$ is reduced
to relators in $R_{G}$.

$t^{i}$ conjugations of $d^{2^{m}}$ are also reduced to $d^{2^{m}}$,
as for each integer $i$ such that $0\le i\le2^{m}-1$,
\begin{align*}
t^{i}d^{2^{m}}t^{-i} & =\left(d^{2^{m}}\right)^{n^{i}}.
\end{align*}

$t^{i}$ conjugations of $tct^{-1}c^{-n}$ and $tdt^{-1}d^{-n}$ are
deleted with removing $2^{m}-1$ additional generators, except for
two relators:
\begin{align*}
d^{n^{2^{m}}-1},\,c^{n^{2^{m}}-1}.
\end{align*}

For any pair of integers $\left(m,n\right)$ such that $m>0$ and
$n$ is odd, $2^{m+2}$ divides $n^{2^{m}}-1$, which is easily proved
by induction. Thus, the relator $d^{n^{2^{m}}-1}$ is deduced from
$d^{2^{m}}$. On the other hand, we decompose $c^{2^{m}}$ as
\begin{align*}
c^{2^{m}} & =\left(b_{0}^{-1}d\right)^{2^{m}}=b_{0}^{-1}\left(db_{0}^{-1}d^{-1}\right)\left(d^{2}b_{0}^{-1}d^{-2}\right)\cdots\left(d^{2^{m}-1}b_{0}^{-1}d^{-2^{m}+1}\right)d^{2^{m}},
\end{align*}

\noindent where the last term is cancelled from $d^{2^{m}}=1$ in
$P_{n,\,m}$. 

Then, $c^{2^{m}}$ is included in the quotient image of
\begin{align*}
\left\langle b_{0},\,db_{0}d^{-1},\,\cdots,\,d^{2^{m}-1}b_{0}d^{-2^{m}+1}\right\rangle _{G}
\end{align*}
in $P_{n,\,m}$. By \citep[Lemma A.3]{lee2023characterizing}, we
have
\begin{align*}
1 & =\left(c^{2^{m}}\right)^{4}=c^{2^{m+2}}
\end{align*}
in $P_{n,\,m}$. From the fact $2^{m+2}\;|\;n^{2^{m}}-1$, this relation
deduced from $d^{2^{m}}=1$ implies $c^{n^{2^{m}}-1}=1$. In summary,
we have established the isomorphism
\begin{align*}
\left\langle \!\!\left\langle c,\,d\right\rangle \!\!\right\rangle _{P_{n,\,m}} & \simeq G/\left\langle \!\!\left\langle d^{2^{m}}\right\rangle \!\!\right\rangle _{G}=Q_{2^{m}}.\;\qedhere
\end{align*}
\end{proof}

\noindent \begin{lemma}

The group $H_{n}$ is a semidirect product satisfying
\[
H_{n}\simeq\widetilde{NC}_{n}\rtimes BS\left(1,\,n\right),
\]
where $BS\left(1,n\right)$ is the Baumslag\textendash Solitar group.

\noindent \end{lemma}
\begin{proof}

Lemma 2.5 and the fact $H_{n}=\widetilde{G}_{n}\rtimes\left\langle t\right\rangle _{H_{n}}$
in the proof of Theorem 1.2 (see Section 2) imply that
\begin{align*}
H_{n} & =\widetilde{G}_{n}\rtimes\left\langle t\right\rangle _{H_{n}}\simeq\left(\widetilde{NC}_{n}\rtimes\mathbb{Z}\left[\frac{1}{n}\right]\right)\rtimes\left\langle t\right\rangle _{H_{n}},
\end{align*}

It suffices to show that $\widetilde{NC}_{n}$ is normal in $H_{n}$,
which means
\begin{align*}
H_{n} & \simeq\left(\widetilde{NC}_{n}\rtimes\mathbb{Z}\left[\frac{1}{n}\right]\right)\rtimes\left\langle t\right\rangle _{H_{n}}\simeq\widetilde{NC}_{n}\rtimes\left(\mathbb{Z}\left[\frac{1}{n}\right]\rtimes\left\langle t\right\rangle _{H_{n}}\right)\simeq\widetilde{NC}_{n}\rtimes BS\left(1,n\right).
\end{align*}

For any pair of integers $\left(i,j\right)$ such that $i\ge0$, take
a generator $t^{-i}b_{j}t^{i}$ of $\widetilde{NC}_{n}$. For any
integer $l$, a computation yields
\begin{align*}
d^{l}t^{-i}b_{j}t^{i}d^{-l} & =t^{-i}\left(t^{i}d^{l}t^{-i}\right)b_{j}\left(t^{i}d^{l}t^{-i}\right)^{-1}t^{i}=t^{-i}d^{ln^{i}}b_{j}d^{-ln^{i}}t^{i}\in\widetilde{NC}_{n}.\;\qedhere
\end{align*}
\end{proof}
\begin{proofthm15}

Moldavanski\u{\i} \citep[Teorema 2]{MR3076943} showed that the Baumslag\textendash Solitar
group $BS\left(1,n\right)$ is residually $p$ if and only if $n\equiv1$
mod $p$. By Lemma 4.5, it suffices to show that for every odd integer
$n$, $H_{n}$ is a residually 2 group.

Fix an odd integer $n$ and select a nonzero element $g$ in $H_{n}$.
According to Lemma 4.5, there exist $\nu\in\widetilde{NC}_{n}$ and
$\beta\in\left\langle d,t\right\rangle _{H_{n}}$ such that $g=\nu\beta$.
If $\beta\ne1$, by using the fact that $\widetilde{NC}_{n}$ is normal
in $H_{n}$, apply the fact that $BS\left(1,n\right)$ is residually
$2$ in \citep{MR3076943} to construct a required homomorphism from
$H_{n}$ to a 2-group. Therefore, we may assume $g\in\widetilde{NC}_{n}$.
Due to the normalcy of $\widetilde{NC}_{n}$, we may further assume
$g\in\left\langle b_{i}\;:\;i=0,1,2,\cdots\right\rangle $, composing
inner automorphisms with the quotient map if necessary.

For some sufficiently large integer $M$, suppose $g\in\left\langle b_{0},b_{1},\cdots,b_{M}\right\rangle $.
Choose an integer $r$ such that $M<2^{r}$. Then, we have
\begin{align*}
\pi_{n,\,r+1}\left(g\right) & \ne1,
\end{align*}
since
\begin{align*}
\xi_{2^{r}} & :\left\langle b_{0},\,b_{1},\,\cdots,\,b_{2^{r}-1}\right\rangle _{G}\to\left\langle b_{0},\,b_{1},\,\cdots,\,b_{2^{r+1}-1}\right\rangle _{Q_{2^{r+1}}}
\end{align*}
is injective by Corollary 4.3. $\qed$

\noindent \end{proofthm15}
\begin{remark}

Lubotzky in \citep{MR0928062} established a well-known theorem stating
that a group $\Gamma$ is linear over a field of characteristic 0
if and only if there is a prime $p$ such that $\Gamma$ admits a
$p$-congruence structure with a bounded rank $c$. According to the
proof of Theorem 1.5, $\left\{ \ker\pi_{n,\,m}\right\} _{m\ge1}$
is a descending chain of normal subgroups of $H_{n}$ such that
\begin{align*}
\bigcap_{m=1}^{\infty}\ker\pi_{n,\,m} & =\left\{ 1\right\} .
\end{align*}

Moreover, for all $m\ge1$, each quotient $\ker\pi_{n,\,1}/\ker\pi_{n,\,m}$
is a finite 2-group. Therefore, we observe that $\left\{ \ker\pi_{n,\,m}\right\} _{m\ge1}$
behaves almost like a 2-congruence structure for $H_{n}$, after formally
adding $H_{n}$ itself as the first element in the chain. When $\left|n\right|>1$,
the only obstruction is the nonexistence of a uniform bound on the
ranks of $\ker\pi_{n,\,m_{1}}/\ker\pi_{n,\,m_{2}}$; if such a bound
existed, $H_{n}$ would be linear, violating Theorem 1.3.

\noindent \end{remark}
\begin{appendices}
	\section{Proof of Lemma 2.1}

\noindent The goal is to prove:

\noindent \begin{lemma21}

The group $NC$ is presented by the generators $\left\{ b_{i}\::\:i\in\mathbb{Z}\right\} $
and the following relations:
\begin{lyxlist}{00.00.0000}
\item [{$\left.1_{b}\right)$}] \noindent $1=b_{i}^{4},\;i\in\mathbb{Z},$
\item [{$\left.2_{b}\right)$}] \noindent $\left[b_{0},b_{i}\right]=b_{i}^{2},\;i\ne0,$
\item [{$\left.3_{b}\right)$}] \noindent $\left[b_{i},b_{j}\right]=b_{i}^{2}b_{j-i}^{2}b_{j}^{2},\;i<j,\;i\ne0\ne j,$
\item [{$\left.4_{b}\right)$}] \noindent $b_{i}^{2}=b_{-i}^{2},\;i>0.$
\end{lyxlist}
\noindent \end{lemma21}
\begin{proof}

From the fact $\det b_{0}=1$ and (1), we apply \citep[Lemma A.2]{lee2023characterizing}
to $\left\langle b_{0},b_{1},\cdots,b_{n}\right\rangle _{G}$ for
each positive integer $n$. Then, this group is presented by the generators
$b_{0},b_{1},\cdots,b_{n}$ and the relations:
\begin{lyxlist}{00.00.0000}
\item [{$\left.1_{b,\,n}\right)$}] \noindent $1=b_{i}^{4},\;0\le i\le n,$
\item [{$\left.2_{b,\,n}\right)$}] \noindent $\left[b_{0},b_{i}\right]=b_{i}^{2},\;0<i\le n,$
\item [{$\left.3_{b,\,n}\right)$}] \noindent $\left[b_{i},b_{j}\right]=b_{i}^{2}b_{j-i}^{2}b_{j}^{2},\;0<i<j\le n.$
\end{lyxlist}
$\quad\;\:$Therefore, the subgroup $\left\langle d^{m}b_{0}d^{-m}\;:\;m\le0\right\rangle $
is presented by the generators $b_{0},b_{1},\cdots$ and the relations:
\begin{lyxlist}{00.00.0000}
\item [{$\left.1_{b,\,+}\right)$}] \noindent $1=b_{i}^{4},\;0\le i,$
\item [{$\left.2_{b,\,+}\right)$}] \noindent $\left[b_{0},b_{i}\right]=b_{i}^{2},\;0<i,$
\item [{$\left.3_{b,\,+}\right)$}] \noindent $\left[b_{i},b_{j}\right]=b_{i}^{2}b_{j-i}^{2}b_{j}^{2},\;0<i<j.$
\end{lyxlist}
$\quad\;\:$Define $NC_{0}:=\left\langle d^{m}b_{0}d^{-m}\;:\;m\le0\right\rangle $
and $NC_{k+1}:=dNC_{k}d^{-1}$ for each $k\ge0$. Then, by definition,
we have
\begin{align*}
NC & =\cup_{k=0}^{\infty}NC_{k}.
\end{align*}
\begin{claim2}

For each integer $k\ge0$, the subgroup $NC_{k}$ is presented by
the generators
\begin{align*}
b_{-k},\,b_{-k+1},\,\cdots,\,b_{0},\,b_{1},\,\cdots
\end{align*}
and the relations:
\begin{lyxlist}{00.00.0000}
\item [{$\left.1_{b,\,-k}\right)$}] \noindent $1=b_{i}^{4},\;-k\le i,$
\item [{$\left.2_{b,\,-k}\right)$}] \noindent $\left[b_{0},b_{i}\right]=b_{i}^{2},\;i\ne0,\;-k\le i,$
\item [{$\left.3_{b,\,-k}\right)$}] \noindent $\left[b_{i},b_{j}\right]=b_{i}^{2}b_{j-i}^{2}b_{j}^{2},\;-k\le i<j,\;i\ne0\ne j,$
\item [{$\left.4_{b,\,-k}\right)$}] \noindent $b_{i}^{2}=b_{-i}^{2},\;0<i\le k.$
\end{lyxlist}
\noindent \end{claim2}

Let us postpone the proof of Claim 2 for a while. Then, the rest is
canonical. Consider a relator $R$ in $NC$. Then, there exists an
integer $k$ such that $R\in NC_{k}$, which implies $R$ is included
in the normal closure of the relators of $NC_{k}$. Therefore, by
collecting all of the relators of $\left.1_{b,\,-k}\right)\text{\textendash}\left.4_{b,\,-k}\right)$
in Claim 2 for every $k$, we obtain the set of relators desired.
$\qedhere$

\noindent \end{proof}
\begin{lemma}

Under the assumptions $\left.1_{b,\,-k}\right)\text{\textendash}\left.4_{b,\,-k}\right)$,
the center of $NC_{k}=$$\left\langle b_{i}\;:\;-k\le i\right\rangle $
contains
\begin{align*}
\left\langle b_{i}^{2}\;:\;-k\le i\right\rangle _{NC_{k}} & =\left\langle b_{i}^{2}\;:\;0\le i\right\rangle _{NC_{k}}.
\end{align*}
\end{lemma}
\begin{proof}

The eqality directly follows from $\left.4_{b,\,-k}\right)$. \citep[Lemma A.3]{lee2023characterizing}
imiplies
\begin{align}
\left[b_{i}^{2},\,b_{j}\right] & =1=\left[b_{i},\,b_{j}^{2}\right],\;0\le i,j.
\end{align}

It suffices to show $\left[b_{i},b_{j}^{2}\right]=1$ when $i<0\le j$.
The cases when $j=0$ are directly established from $\left.2_{b,\,-k}\right)$.
When $i<0<j$, we deduce that
\begin{align*}
 & \,\left[b_{i},\,b_{j}^{2}\right]=b_{i}^{-1}b_{j}^{-2}b_{i}b_{j}^{2}\\
= & \,b_{i}^{-1}b_{j}^{-1}b_{i}\left[b_{i},\,b_{j}\right]b_{j}\\
= & \,b_{i}^{-1}b_{j}^{-1}b_{i}^{-1}b_{j-i}^{2}b_{j}^{-1}\\
= & \,\left[b_{i},\,b_{j}\right]b_{j}^{-1}b_{i}^{-2}b_{j-i}^{2}b_{j}^{-1}\\
= & \,b_{i}^{2}b_{j-i}^{2}b_{j}b_{i}^{-2}b_{j-i}^{2}b_{j}^{-1}\\
= & \,b_{-i}^{2}b_{j-i}^{2}b_{j}b_{-i}^{-2}b_{j-i}^{2}b_{j}^{-1}\\
= & \,1,
\end{align*}

\noindent where the third equality follows from $\left.3_{b,\,-k}\right)$;
the fifth from $\left.3_{b,\,-k}\right)$; the sixth from $\left.4_{b,\,-k}\right)$;
the last follows from $\left.1_{b,\,-k}\right)$ and (4). $\qedhere$

\noindent \end{proof}
\begin{proofclaim2}

From the construction, $NC_{k}$ and $NC_{k+1}$ is isomorphic for
every nonnegative integer $k$, with the isomorphism
\begin{align*}
\gamma_{k} & :NC_{k}\to NC_{k+1},\;\gamma_{k}\left(x\right):=dxd^{-1}.
\end{align*}

We use induction. When $k=0$, it is done for the presentation of
$NC_{0}$. Suppose the presentation of $NC_{k}$ as above. For every
integer $i\ne0,1$, we have
\begin{align*}
d^{-1}b_{i}d & =d^{-1}b_{0}^{-1}d^{-i}b_{0}d^{i+1}=b_{1}^{-1}b_{i+1},
\end{align*}
and for $i=0,1,$
\begin{align*}
 & \,d^{-1}b_{0}d=b_{0}b_{1},\\
 & \,d^{-1}b_{-1}d=d^{-1}b_{0}^{-1}db_{0}=b_{1}^{-1}.
\end{align*}

By applying the isomorphism $\gamma_{k}^{-1}$, it suffices to prove
for $NC_{k}$, the relations $\left.1_{b,\,-k}\right)\text{\textendash}\left.4_{b,\,-k}\right)$
are deduced from
\begin{lyxlist}{00.00.0000}
\item [{$\left.1_{b,\,-k}'\right)$}] \noindent $1=\left(b_{1}^{-1}b_{i}\right)^{4},\;-k\le i,$
\item [{$\left.1_{b}^{-1}\right)$}] \noindent $1=b_{1}^{4},$
\item [{$\left.2_{b,\,-k}'\right)$}] \noindent $\left[b_{0}b_{1},b_{1}^{-1}b_{i}\right]=\left(b_{1}^{-1}b_{i}\right)^{2},\;i\ne0,1,\;-k\le i,$
\item [{$\left.2_{b}^{-1}\right)$}] \noindent $\left[b_{0}b_{1},b_{1}^{-1}\right]=b_{1}^{-2},$
\item [{$\left.3_{b,\,-k}'\right)$}] \noindent $\left[b_{1}^{-1}b_{i},b_{1}^{-1}b_{j}\right]=\left(b_{1}^{-1}b_{i}\right)^{2}\left(b_{1}^{-1}b_{j-i+1}\right)^{2}\left(b_{1}^{-1}b_{j}\right)^{2},\;-k\le i<j,\;i\ne0,1\ne j,$
\item [{$\left.3_{b}^{-1},+\right)$}] \noindent $\left[b_{1}^{-1},b_{1}^{-1}b_{j}\right]=b_{1}^{-2}\left(b_{1}^{-1}b_{j+1}\right)^{2}\left(b_{1}^{-1}b_{j}\right)^{2},\;1<j,$
\item [{$\left.3_{b,\,-k}^{-1},-\right)$}] \noindent $\left[b_{1}^{-1}b_{i},b_{1}^{-1}\right]=\left(b_{1}^{-1}b_{i}\right)^{2}\left(b_{1}^{-1}b_{-i+1}\right)^{2}b_{1}^{-2},\;-k\le i<0,$
\item [{$\left.4_{b,\,-k}'\right)$}] \noindent $\left(b_{1}^{-1}b_{i+2}\right)^{2}=\left(b_{1}^{-1}b_{-i}\right)^{2},\;0<i\le k,$
\item [{$\left.4_{b}^{-1}\right)$}] \noindent $\left(b_{1}^{-1}b_{2}\right)^{2}=b_{1}^{-2},$
\end{lyxlist}
and vice versa.

Suppose $\left.1_{b,\,-k}\right)\text{\textendash}\left.4_{b,\,-k}\right)$.
Then, $\left.1_{b}^{-1}\right)$ follows from $\left.1_{b,\,-k}\right)$;
$\left.2_{b}^{-1}\right)$ from $\left.2_{b,\,-k}\right)$ by using
the commutator identity
\begin{align}
\left[a,\,bc\right] & =\left[a,\,c\right]c^{-1}\left[a,\,b\right]c;
\end{align}
$\left.3_{b}^{-1},+\right)$ from $\left.3_{b,\,-k}\right)$ and (5);
$\left.3_{b,\,-k}^{-1},-\right)$ from $\left.3_{b,\,-k}\right)$
and (5); $\left.4_{b}^{-1}\right)$ from $\left.3_{b,\,-k}\right)$;
$\left.1_{b,\,-k}'\right)$ from $\left.3_{b,\,-k}\right)$. To obatin
$\left.2_{b,\,-k}'\right)$, expanding the commutator in the left
hand side, we compute that
\begin{align*}
 & \,\left[b_{0}b_{1},\,b_{1}^{-1}b_{i}\right]=\\
= & \,\left[b_{0}b_{1},\,b_{i}\right]\left[b_{0}b_{1},\,b_{1}^{-1}\right]\\
= & \,\left[b_{0},\,b_{i}\right]\left[b_{1},\,b_{i}\right]\left[b_{0},\,b_{1}^{-1}\right]\\
= & \,b_{1}^{2}b_{i}^{2}\left[b_{1},\,b_{i}\right]=b_{1}^{2}b_{i}^{2}\left[b_{i},\,b_{1}\right],
\end{align*}
where $\left[b_{i},b_{j}\right]$ is in the center from $\left.2_{b,\,-k}\right)$,
$\left.3_{b,\,-k}\right)$, and Lemma A.1; the fourth equality follows
from Lemma A.1, $\left.1_{b,\,-k}\right)$, and $\left.2_{b,\,-k}\right)$;
and the last from $\left.1_{b,\,-k}\right)$. By comparison, the right
hand side of $\left.2_{b,\,-k}'\right)$ becomes
\begin{align*}
\left(b_{1}^{-1}b_{i}\right)^{2} & =b_{1}b_{i}b_{1}b_{i}=b_{1}^{2}b_{i}\left[b_{i},\,b_{1}\right]b_{i}=b_{1}^{2}b_{i}^{2}\left[b_{i},\,b_{1}\right],
\end{align*}
which establishes $\left.2_{b,\,-k}'\right)$. The proof of $\left.3_{b,\,-k}'\right)$
is similar, by expanding the left hand side by using (5). For $\left.4_{b,\,-k}'\right)$,
we need to show
\begin{align}
b_{i+2}b_{1}b_{i+2} & =b_{-i}b_{1}b_{-i}.
\end{align}

Under the assumption $i\ne0$, we deduce (6) from $\left.3_{b,\,-k}\right)$,
which concludes the deduction from $\left.1_{b,\,-k}\right)\text{\textendash}\left.4_{b,\,-k}\right)$.
Indeed,
\begin{align*}
 & \,b_{i+2}b_{1}b_{i+2}\\
= & \,b_{i+2}^{2}b_{1}\left[b_{1},\,b_{i+2}\right]=b_{1}^{-1}b_{i+1}^{2}\\
= & \,b_{-i}^{2}b_{1}\left[b_{1},\,b_{-i}\right]=b_{-i}b_{1}b_{-i}.
\end{align*}

On the other hand, suppose $\left.1_{b,\,-k}'\right)\text{\textendash}\left.4_{b,\,-k}'\right)$,
$\left.1_{b}^{-1}\right)$, $\left.2_{b}^{-1}\right)$, $\left.3_{b}^{-1},+\right)$,
$\left.3_{b,\,-k}^{-1},-\right)$, $\left.4_{b}^{-1}\right)$. At
first, the relation $\left.1_{b,\,-k}\right)$ for $i=1$ directly
follows from $\left.1_{b}^{-1}\right)$. By applying (5) to $\left.2_{b}^{-1}\right)$
and by using $\left.1_{b}^{-1}\right)$, we have $\left.2_{b,\,-k}\right)$
for $i=1$ from
\begin{align*}
b_{1}^{-2} & =\left[b_{0}b_{1},\,b_{1}^{-1}\right]=b_{1}^{-1}\left[b_{0},\,b_{1}^{-1}\right]b_{1}=\left[b_{1},\,b_{0}\right].
\end{align*}

By applying $\left.1_{b}^{-1}\right)$ and $\left.2_{b,\,-k}\right)$
for $i=1$, we have $\left.1_{b,\,-k}\right)$ for $i=0$.

For $i\ne0,1$, by expanding the left hand side of $\left.2_{b,\,-k}'\right)$
we deduce that
\begin{align*}
 & \,\left(b_{1}^{-1}b_{i}\right)^{2}=\left[b_{0}b_{1},\,b_{1}^{-1}b_{i}\right]\\
= & \,\left[b_{0}b_{1},\,b_{i}\right]b_{i}^{-1}\left[b_{0}b_{1},\,b_{1}^{-1}\right]b_{i}\\
= & \,\left[b_{0}b_{1},\,b_{i}\right]b_{i}^{-1}b_{1}^{-2}b_{i}\\
= & \,b_{1}^{-1}\left[b_{0},\,b_{i}\right]b_{1}\left[b_{1},\,b_{i}\right]b_{i}^{-1}b_{1}^{-2}b_{i},
\end{align*}
where after equating the first term and the last, by reducing the
same words, we deduce that
\begin{align*}
 & \,b_{i}=\left[b_{0},\,b_{i}\right]b_{1}\left[b_{1},\,b_{i}\right]b_{i}^{-1}b_{1}^{-1},\\
\iff & \,b_{i}^{2}b_{1}\left(b_{1}^{-1}b_{i}^{-1}b_{1}b_{i}\right)=\left[b_{0},\,b_{i}\right]b_{1}\left[b_{1},\,b_{i}\right],\\
\iff & \,b_{i}^{2}=\left[b_{0},\,b_{i}\right],
\end{align*}
so finally we have established $\left.2_{b,\,-k}\right)$ for every
$i$.

For convenience, put $c_{i}:=b_{1}^{-1}b_{i}$ for each $i$. We show
$\left.1_{b,\,-k}\right)$ for every $i$. The cases $i=0,1$ are
already done. For $i>1$, we expand
\begin{align*}
 & \,b_{i}^{4}=\left(b_{1}c_{i}b_{1}c_{i}\right)^{2}\\
= & \,\left(c_{i}\left[c_{i},\,b_{1}^{-1}\right]b_{1}^{2}c_{i}\right)^{2}\\
= & \,\left(c_{i}^{-1}c_{i+1}^{-2}b_{1}^{4}c_{i}\right)^{2}\\
= & \,c_{i}^{-1}c_{i+1}^{-4}c_{i}=1,
\end{align*}
where the third equality follows from $\left.3_{b}^{-1},+\right)$,
the fourth from $\left.1_{b}^{-1}\right)$, and the last from $\left.1_{b,\,-k}'\right)$.
For $i$ such that $-k\le i<0$, we expand
\begin{align*}
 & \,b_{i}^{4}=\left(c_{i}\left[c_{i},\,b_{1}^{-1}\right]b_{1}^{2}c_{i}\right)^{2}\\
= & \,\left(c_{i}^{3}c_{1-i}^{2}c_{i}\right)^{2}\\
= & \,c_{i}^{3}c_{1-i}^{2}c_{i}^{4}c_{1-i}^{2}c_{i}=1,
\end{align*}
where the second eqality follows from $\left.3_{b,\,-k}^{-1},-\right)$,
and the last from $\left.1_{b,\,-k}'\right)$.

To deal with $\left.3_{b,\,-k}\right)$ and $\left.4_{b,\,-k}\right)$,
we introduce several claims.

\noindent \begin{claim3}

Under the assumptions $\left.1_{b,\,-k}'\right)\text{\textendash}\left.4_{b,\,-k}'\right)$,
$\left.1_{b}^{-1}\right)$, $\left.2_{b}^{-1}\right)$, $\left.3_{b}^{-1},+\right)$,
$\left.3_{b,\,-k}^{-1},-\right)$, $\left.4_{b}^{-1}\right)$, we
have
\begin{align}
 & \,c_{j}^{2}\in Z\left(\left\langle c_{0},\,c_{2},\,c_{3},\cdots\right\rangle \right),\;j\ge0.\\
 & \,\left[b_{1}^{2},\,c_{j}\right]=1,\;j\ge0,\\
 & \,\left[b_{1},\,c_{j}^{2}\right]=1,\;j\ge0,
\end{align}
where $Z\left(A\right)$ is the center of the group $A$.

\noindent \end{claim3}
\begin{proofclaim3}

We first show (7) by induction. When $j=0$, for any integer $i$
such that $1<i$, the relation $\left.2_{b,\,-k}'\right)$ $\left[c_{0}^{-1},c_{i}\right]=c_{i}^{2}$
implies
\begin{align*}
 & \,\left[c_{0}^{2},\,c_{i}\right]=c_{0}^{-2}c_{i}^{-1}c_{0}^{2}c_{i}\\
= & \,c_{0}^{-1}\left[c_{0},\,c_{i}\right]c_{i}^{-1}c_{0}c_{i}\\
= & \,c_{0}^{-1}c_{i}c_{0}c_{i}=c_{i}\left[c_{i},\,c_{0}\right]c_{i}=1.
\end{align*}

For any pair $\left(i,j\right)$ such that $1<i<j$, the relation
$\left.3_{b,\,-k}'\right)$ means $\left[c_{i},c_{j}\right]=c_{i}^{2}c_{j-i+1}^{2}c_{j}^{2}$.
This implies $\left[c_{2},c_{3}\right]=c_{3}^{2}$ or $c_{2}^{-1}c_{3}^{-1}c_{2}=c_{3}$.
Therefore, we have $\left[c_{2},c_{3}^{2}\right]=1=\left[c_{2}^{2},c_{3}\right]$.
For some integer $4\le N$, suppose $\left[c_{i}^{2},c_{j}\right]=1$
for every pair $\left(i,j\right)$ such that $1<i,j<N$. By expanding
$\left.3_{b,\,-k}'\right)$ we have
\begin{align*}
 & \,\left[c_{i},\,c_{N}\right]=c_{i}^{2}c_{N-i+1}^{2}c_{N}^{2}\\
\iff & \,\left(c_{i}c_{N}^{-1}\right)^{2}=c_{N-i+1}^{2},
\end{align*}
which implies
\begin{align*}
 & \,1=\left[c_{N-i+1}^{2},\,c_{i}c_{N}^{-1}\right]\\
\iff & \,1=\left[c_{N-i+1}^{2},\,c_{N}^{-1}\right]c_{N}\left[c_{N-i+1}^{2},\,c_{i}\right]c_{N}^{-1}\\
\iff & \,1=\left[c_{N-i+1}^{2},\,c_{N}^{-1}\right],
\end{align*}
where the first equivalence follows from (5) and the second from the
induction hypothesis. From this calculation, we have established
\begin{align}
1 & =\left[c_{i}^{2},\,c_{N}\right],\;1<i<N.
\end{align}

On the other hand, we also have
\begin{align*}
 & \,1=\left[c_{N}^{2},\,c_{i}\right]\\
\iff & \,1=c_{N}^{-2}c_{i}^{-1}c_{N}^{2}c_{i}\\
\iff & \,1=c_{N}^{-1}\left[c_{N},\,c_{i}\right]c_{i}^{-1}c_{N}c_{i}\\
\iff & \,1=c_{N}c_{N-i+1}^{2}c_{i}c_{N}c_{i}
\end{align*}
\begin{align*}
\iff & \,c_{i}c_{N}c_{i}c_{N}=c_{N-i+1}^{-2}\\
\iff & \,c_{i}^{-1}c_{N}c_{i}^{-1}c_{N}=c_{N-i+1}^{-2}\\
\iff & \,\left[c_{i},\,c_{N}\right]=c_{i}^{2}c_{N-i+1}^{2}c_{N}^{2},
\end{align*}
where the third equivalence follows from $\left.3_{b,\,-k}'\right)$;
the fourth from (10); and the fifth from (10) and $\left.1_{b,\,-k}'\right)$.
It concludes the proof of (7).

The equation (8) is direct from $\left.4_{b}^{-1}\right)$ $c_{2}^{2}=b_{1}^{-2}$
and (7). For (9), for each $i>1$, we expand
\begin{align*}
 & \,\left[b_{1},\,c_{j}^{2}\right]\\
= & \,\left[b_{1},\,c_{j}\right]c_{j}^{-1}\left[b_{1},\,c_{j}\right]c_{j}\\
= & \,b_{1}^{-2}c_{j+1}^{2}c_{j}^{2}c_{j}^{-1}b_{1}^{-2}c_{j+1}^{2}c_{j}^{3},\\
= & \,b_{1}^{-4}c_{j+1}^{4}c_{j}^{4}=1,
\end{align*}
where the first equality follows from (5); the second from $\left.3_{b}^{-1},+\right)$;
the third from (7) and (8); and the last from $\left.1_{b,\,-k}'\right)$
and $\left.1_{b}^{-1}\right)$. $\qed$ \end{proofclaim3}
\begin{claim4}

Under the assumptions $\left.1_{b,\,-k}'\right)\text{\textendash}\left.4_{b,\,-k}'\right)$,
$\left.1_{b}^{-1}\right)$, $\left.2_{b}^{-1}\right)$, $\left.3_{b}^{-1},+\right)$,
$\left.3_{b,\,-k}^{-1},-\right)$, $\left.4_{b}^{-1}\right)$, we
have
\begin{align}
b_{i}^{2} & \in Z\left(\left\langle b_{0},\,b_{1},\,b_{2},\cdots\right\rangle \right),\;i\ge0.
\end{align}

\noindent \end{claim4}
\begin{proofclaim4}

At first, the cases $i=0,1$ directly follow from $\left.2_{b,\,-k}\right)$
and (8). The equation $\left[b_{0},b_{j}^{2}\right]=1$ for $j>1$
also follows from $\left.1_{b,\,-k}\right)$ and $\left.2_{b,\,-k}\right)$.
For any integer $j>1$, observe that
\begin{align*}
 & \,\left[b_{1},\,b_{j}^{2}\right]\\
= & \,\left[b_{1},\,b_{1}c_{j}b_{1}c_{j}\right]\\
= & \,\left[b_{1},\,b_{1}^{2}c_{j}^{2}\left[c_{j},\,b_{1}\right]\right]\\
= & \,1,
\end{align*}
where the last equality follows from (8), (9), and $\left.3_{b}^{-1},+\right)$.
On the other hand, for integers $i,j\ge1$ we deduce that
\begin{align*}
 & \,\left[b_{i}^{2},\,b_{j}\right]\\
= & \,\left[b_{1}c_{i}b_{1}c_{i},\,b_{1}c_{j}\right]\\
= & \,\left[b_{1}^{2}c_{i}^{2}\left[b_{1},\,c_{i}\right],\,b_{1}c_{j}\right]\\
= & \,\left[b_{1}^{2}c_{i}^{2}\left[b_{1},\,c_{i}\right],\,c_{j}\right]c_{j}^{-1}\left[b_{1}^{2}c_{i}^{2}\left[b_{1},\,c_{i}\right],\,b_{1}\right]c_{j}\\
= & \,1,
\end{align*}
where the last equality follows from (8), (9), and $\left.3_{b}^{-1},+\right)$.
$\qed$

\noindent \end{proofclaim4}
\begin{claim5}

Under the assumptions $\left.1_{b,\,-k}'\right)\text{\textendash}\left.4_{b,\,-k}'\right)$,
$\left.1_{b}^{-1}\right)$, $\left.2_{b}^{-1}\right)$, $\left.3_{b}^{-1},+\right)$,
$\left.3_{b,\,-k}^{-1},-\right)$, $\left.4_{b}^{-1}\right)$, when
$1<j$, we have
\begin{align}
 & \,\left[b_{1},\,b_{j}\right]\,=b_{1}^{2}b_{j-1}^{2}b_{j}^{2},\\
 & \,c_{j}^{2}\,=b_{j-1}^{2}.
\end{align}
\end{claim5}
\begin{proofclaim5}

We use induction. The case $j=2$ follows from $\left.4_{b}^{-1}\right)$.
Suppose (12) holds for $1<j\le N$. By the induction hypothesis, we
have
\begin{align*}
 & \,b_{1}^{2}b_{N-1}^{2}b_{N}^{2}\\
= & \,\left[b_{1},\,b_{N}\right]=\left[b_{1}^{-1},\,c_{N}\right]\\
= & \,b_{1}^{-2}\left(b_{1}^{-1}b_{N+1}\right)^{2}\left(b_{1}^{-1}b_{N}\right)^{2}\\
= & \,b_{1}\left(b_{N+1}b_{1}^{-1}b_{N+1}b_{1}^{-1}\right)\left(b_{N}b_{1}^{-1}b_{N}\right),
\end{align*}
where the third equality follows from $\left.3_{b}^{-1},+\right)$.
By equating the first term with the last term, we establish (12) as
follows:
\begin{align*}
 & \,b_{1}b_{N-1}^{2}b_{N}b_{1}b_{N}^{-1}=b_{N+1}b_{1}^{-1}b_{N+1}b_{1}^{-1}\\
\iff & \,b_{1}b_{N-1}^{2}b_{N}^{2}b_{1}\left[b_{1},\,b_{N}\right]b_{N}^{2}=b_{N+1}b_{1}^{-1}b_{N+1}b_{1}^{-1}\\
\iff & \,b_{N}^{2}=b_{N+1}b_{1}^{-1}b_{N+1}b_{1}^{-1}\\
\iff & \,b_{N+1}^{2}b_{N}^{2}b_{1}^{2}=\left[b_{N+1},\,b_{1}\right],
\end{align*}
where the last equivalence follows from the induction hypothesis and
(9).

The equation (13) directly follows from (12). Indeed,
\begin{align*}
c_{i}^{2} & =\left(b_{1}^{-1}b_{i}\right)^{2}=b_{1}b_{i}b_{1}b_{i}=b_{1}b_{i}^{2}b_{1}\left[b_{1},\,b_{i}\right]=b_{i-1}^{2}.\;\qed
\end{align*}
\end{proofclaim5}
\begin{claim6}

Under the assumptions $\left.1_{b,\,-k}'\right)\text{\textendash}\left.4_{b,\,-k}'\right)$,
$\left.1_{b}^{-1}\right)$, $\left.2_{b}^{-1}\right)$, $\left.3_{b}^{-1},+\right)$,
$\left.3_{b,\,-k}^{-1},-\right)$, $\left.4_{b}^{-1}\right)$, when
$-k\le i<0$, we have
\begin{align}
 & \,c_{i}^{2}\,=b_{1-i}^{2}.\\
 & \,\left[b_{i},\,b_{1}^{-1}\right]\,=b_{1-i}^{2}b_{-i}^{2}b_{1}^{2}.
\end{align}
\end{claim6}
\begin{proofclaim6}

The equation (14) directly follows from $\left.4_{b,\,-k}'\right)$
and (13). By expanding the left hand side of (15), we have
\begin{align*}
 & \,\left[b_{i},\,b_{1}^{-1}\right]=\left[c_{i},\,b_{1}^{-1}\right]\\
= & \,\left(b_{1}^{-1}b_{i}\right)^{2}\left(b_{1}^{-1}b_{1-i}\right)^{2}b_{1}^{-2}\\
= & \,b_{1-i}^{2}b_{-i}^{2}b_{1}^{2},
\end{align*}
where the second equality follows from $\left.3_{b,\,-k}^{-1},-\right)$,
and the last from (14). $\qed$

\noindent \end{proofclaim6}

We return to the proof of Claim 2. To show $\left.4_{b,\,-k}\right)$,
for any integer $i$ such that $0<i\le k$, by expanding $\left.4_{b,\,-k}'\right)$
we deduce that
\begin{align*}
 & \,\left(b_{1}^{-1}b_{i+2}\right)^{2}=\left(b_{1}^{-1}b_{-i}\right)^{2}\\
\iff & \,b_{i+1}^{2}=b_{1}^{-1}b_{-i}b_{1}^{-1}b_{-i}\\
\iff & \,b_{1}b_{i+1}^{2}=b_{-i}b_{1}^{-1}b_{-i}\\
\iff & \,b_{1}b_{i+1}^{2}=b_{-i}^{2}b_{1}^{-1}\left[b_{1}^{-1},\,b_{-i}\right]\\
\iff & \,b_{1}b_{i+1}^{2}=b_{-i}^{2}b_{i+1}^{2}b_{i}^{2}b_{1}\\
\iff & \,b_{1}b_{i+1}^{2}b_{1}^{-1}b_{i}^{-2}b_{i+1}^{-2}=b_{-i}^{2}\\
\iff & \,b_{i}^{2}=b_{-i}^{2},
\end{align*}
where the first equivalence follows from (13); the fourth from (11),
(15), and $\left.1_{b,\,-k}\right)$; and the last from (11) and $\left.1_{b,\,-k}\right)$.

Finally, we show $\left.3_{b,\,-k}\right)$. The case $1=i<j$ is
already handled by (12). The case $i<j=1$ also directly follows from
(15) and $\left.4_{b,\,-k}\right)$. Therefore, we may assume $i\ne1\ne j$.
For any pair $\left(i,j\right)$ such that $1<i<j$, by expanding
the left hand side of $\left.3_{b,\,-k}\right)$, we have
\begin{align*}
 & \,\left[b_{i},\,b_{j}\right]=\left[b_{1}c_{i},\,b_{1}c_{j}\right]\\
= & \,\left[b_{1}c_{i},\,c_{j}\right]c_{j}^{-1}\left[b_{1}c_{i},\,b_{1}\right]c_{j}\\
= & \,c_{i}^{-1}\left[b_{1},\,c_{j}\right]c_{i}\left[c_{i},\,c_{j}\right]c_{j}^{-1}\left[c_{i},\,b_{1}\right]c_{j}\\
= & \,c_{i}^{-1}\left[b_{1},\,b_{j}\right]c_{i}^{3}c_{j-i+1}^{2}c_{j}\left[b_{i},\,b_{1}\right]c_{j}\\
= & \,c_{i}^{2}\left(b_{1}^{2}b_{j-1}^{2}b_{j}^{2}\right)\left(b_{1}^{2}b_{i-1}^{2}b_{i}^{2}\right)c_{j}^{2}c_{j-i+1}^{2}\\
= & \,b_{i}^{2}b_{j-i}^{2}b_{j}^{2},
\end{align*}
where the fourth equality follows from $\left.3_{b,\,-k}'\right)$,
the fifth from (12), and the last from (13), (9), and $\left.1_{b,\,-k}\right)$.
For the cases $i<0$, we need to introduce more claims.

\noindent \begin{claim7}

Under the assumptions $\left.1_{b,\,-k}'\right)\text{\textendash}\left.4_{b,\,-k}'\right)$,
$\left.1_{b}^{-1}\right)$, $\left.2_{b}^{-1}\right)$, $\left.3_{b}^{-1},+\right)$,
$\left.3_{b,\,-k}^{-1},-\right)$, $\left.4_{b}^{-1}\right)$, for
any pair of integers $\left(i,j\right)$ such that $i<0$ and $j>0$,
$\left[b_{i},b_{j}\right]$ is included in the abelian subgroup $\left\langle b_{0}^{2},b_{1}^{2},\cdots\right\rangle .$

\noindent \end{claim7}
\begin{proofclaim7}

As in the partial proof of $\left.3_{b,\,-k}\right)$, we compute
that
\begin{align*}
 & \,\left[b_{i},\,b_{j}\right]\\
= & \,c_{i}^{-1}\left[b_{1},\,b_{j}\right]c_{i}^{3}c_{j-i+1}^{2}c_{j}\left[b_{i},\,b_{1}\right]c_{j}\\
= & \,c_{i}^{-1}\left(b_{1}^{2}b_{j-1}^{2}b_{j}^{2}\right)c_{i}^{3}b_{j-i}^{2}c_{j}b_{1-i}^{2}b_{-i}^{2}b_{1}^{2}c_{j}\\
= & \,c_{i}^{-1}\left(b_{1}^{2}b_{j-1}^{2}b_{j}^{2}\right)c_{i}^{3}b_{j-i}^{2}b_{1-i}^{2}b_{-i}^{2}b_{1}^{2}b_{j-1}^{2},
\end{align*}
where the second equality follows from (12) and (15); the last from
(11) and (13).

Therefore, it suffices to prove for each $\left(i,j\right)$ such
that $i<0$ and $j>0$, $c_{i}^{-1}b_{j}^{2}c_{i}$ is included in
the abelian subgroup $\left\langle b_{0}^{2},b_{1}^{2},\cdots\right\rangle .$
By expanding it,
\begin{align*}
 & \,c_{i}^{-1}b_{j}^{2}c_{i}=b_{j}^{2}\left[b_{j}^{2},\,c_{i}\right]\\
= & \,b_{j}^{2}\left[b_{1}c_{j}b_{1}c_{j},\,c_{i}\right]\\
= & \,b_{j}^{2}c_{j}^{-1}\left[b_{1}c_{j}b_{1},\,c_{i}\right]c_{j}\left[c_{j},\,c_{i}\right]\\
= & \,b_{j}^{2}c_{j}^{-1}b_{1}^{-1}\left[b_{1}c_{j},\,c_{i}\right]b_{1}\left[b_{1},\,c_{i}\right]c_{j}\left[c_{j},\;c_{i}\right]\\
= & \,b_{j}^{2}c_{j}^{-1}b_{1}^{-1}c_{j}^{-1}\left[b_{1},\;c_{i}\right]c_{j}\left[c_{j},\;c_{i}\right]b_{1}\left[b_{1},\;c_{i}\right]c_{j}\left[c_{j},\;c_{i}\right],
\end{align*}
where in the last term, every commutators are in the $\left\langle b_{0}^{2},b_{1}^{2},\cdots\right\rangle $
from $\left.1_{b}^{-1}\right)$, $\left.3_{b,\,-k}'\right)$, $\left.3_{b,\,-k}^{-1},-\right)$,
(13), and (14). By rearranging the last term, we have
\begin{align*}
c_{i}^{-1}b_{j}^{2}c_{i} & =b_{j}^{2}.\;\qed
\end{align*}
\end{proofclaim7}
\begin{claim8}

Under the assumptions $\left.1_{b,\,-k}'\right)\text{\textendash}\left.4_{b,\,-k}'\right)$,
$\left.1_{b}^{-1}\right)$, $\left.2_{b}^{-1}\right)$, $\left.3_{b}^{-1},+\right)$,
$\left.3_{b,\,-k}^{-1},-\right)$, $\left.4_{b}^{-1}\right)$, for
any pair of integers $\left(i,j\right)$ such that $i<0$ and $j\ge0$,
we have
\begin{align}
 & \,\left[b_{i},\,b_{j}^{2}\right]=1,
\end{align}
\end{claim8}
\begin{proofclaim8}

It is direct from Claim 7. Indeed,
\begin{align*}
b_{i}b_{j}^{2} & =b_{j}b_{i}\left[b_{i},\,b_{j}\right]b_{j}=b_{j}b_{i}b_{j}\left[b_{i},\,b_{j}\right]=b_{j}^{2}b_{i}\left[b_{i},\,b_{j}\right]^{2}=b_{j}^{2}b_{i},
\end{align*}
\end{proofclaim8}

We again return to the proof of $\left.3_{b,\,-k}\right)$. Suppose
$i<0$ and $1<j$. We expand
\begin{align*}
 & \,\left[b_{i},\,b_{j}\right]\\
= & \,c_{i}^{-1}\left(b_{1}^{2}b_{j-1}^{2}b_{j}^{2}\right)c_{i}^{3}b_{j-i}^{2}b_{1-i}^{2}b_{-i}^{2}b_{1}^{2}b_{j-1}^{2}\\
= & \,\left(b_{1}^{2}b_{j-1}^{2}b_{j}^{2}\right)b_{1-i}^{2}b_{j-i}^{2}b_{1-i}^{2}b_{-i}^{2}b_{1}^{2}b_{j-1}^{2}\\
= & \,b_{i}^{2}b_{j-i}^{2}b_{j}^{2},
\end{align*}
where the first equality follows from the proof of Claim 7; the second
from (11), (14), and (16); and the last from (11), $\left.1_{b,\,-k}\right)$,
and $\left.4_{b,\,-k}\right)$.

On the other hand, suppose $i<j<0$. We expand
\begin{align*}
 & \,\left[b_{i},\,b_{j}\right]\\
= & \,c_{i}^{-1}\left[b_{1},\,b_{j}\right]c_{i}^{3}c_{j-i+1}^{2}c_{j}\left[b_{i},\,b_{1}\right]c_{j}\\
= & \,c_{i}^{-1}b_{1-j}^{2}b_{-j}^{2}b_{1}^{2}c_{i}^{3}b_{j-i}^{2}c_{j}b_{1-i}^{2}b_{-i}^{2}b_{1}^{2}c_{j}\\
= & \,b_{1-j}^{2}b_{-j}^{2}b_{1}^{2}c_{i}^{2}b_{j-i}^{2}b_{1-i}^{2}b_{-i}^{2}b_{1}^{2}c_{j}^{2}\\
= & \,b_{1-j}^{2}b_{-j}^{2}b_{1}^{2}b_{1-i}^{2}b_{j-i}^{2}b_{1-i}^{2}b_{-i}^{2}b_{1}^{2}b_{1-j}^{2}\\
= & \,b_{i}^{2}b_{j-i}^{2}b_{j}^{2},
\end{align*}
where the second equality from (15); the third from (16); the fourth
from (14); and the last from (11), $\left.1_{b,\,-k}\right)$, and
$\left.4_{b,\,-k}\right)$. We have established $\left.3_{b,\,-k}\right)$,
and finished the proof of Claim 2. $\qed$

\noindent \end{proofclaim2}

\noindent \end{appendices}

\begin{spacing}{0.9}
\bibliographystyle{amsplain}
\phantomsection\addcontentsline{toc}{section}{\refname}\bibliography{bibgen}

\providecommand{\bysame}{\leavevmode\hbox to3em{\hrulefill}\thinspace}
\providecommand{\MR}{\relax\ifhmode\unskip\space\fi MR }
\providecommand{\MRhref}[2]{%
  \href{http://www.ams.org/mathscinet-getitem?mr=#1}{#2}
}
\providecommand{\href}[2]{#2}
\begin{thebibliography}{10}

\bibitem{MR2138070}
Alexander Borisov and Mark Sapir, \emph{Polynomial maps over finite fields and residual finiteness of mapping tori of group endomorphisms}, Invent. Math. \textbf{160} (2005), no.~2, 341--356. \MR{2138070}

\bibitem{MR2533795}
\bysame, \emph{Polynomial maps over {$p$}-adics and redisual properties of mapping tori of group endomorphisms}, Int. Math. Res. Not. IMRN (2009), no.~16, 3002--3015. \MR{2533795}

\bibitem{MR4388367}
Hip~Kuen Chong and Daniel~T. Wise, \emph{An uncountable family of finitely generated residually finite groups}, J. Group Theory \textbf{25} (2022), no.~2, 207--216. \MR{4388367}

\bibitem{MR1431138}
D.~Cooper and D.~D. Long, \emph{A presentation for the image of {${\rm Burau}(4)\otimes Z_2$}}, Invent. Math. \textbf{127} (1997), no.~3, 535--570. \MR{1431138}

\bibitem{MR2115010}
Cornelia Dru\c{t}u and Mark Sapir, \emph{Non-linear residually finite groups}, J. Algebra \textbf{284} (2005), no.~1, 174--178. \MR{2115010}

\bibitem{MR3671739}
Olga Kharlampovich, Alexei Myasnikov, and Mark Sapir, \emph{Algorithmically complex residually finite groups}, Bull. Math. Sci. \textbf{7} (2017), no.~2, 309--352. \MR{3671739}

\bibitem{lee2023characterizing}
Donsung Lee, \emph{Characterizing the kernel of the {B}urau representation modulo 2 for the 4-strand braid group}, 2023. arXiv:2309.05547.

\bibitem{MR0928062}
Alexander Lubotzky, \emph{A group theoretic characterization of linear groups}, J. Algebra \textbf{113} (1988), no.~1, 207--214. \MR{928062}

\bibitem{MR0003420}
A.~Malcev, \emph{On isomorphic matrix representations of infinite groups (in {R}ussian)}, Rec. Math. [Mat. Sbornik] N.S. \textbf{8/50} (1940), 405--422. \MR{3420}

\bibitem{MR0285589}
Stephen Meskin, \emph{Nonresidually finite one-relator groups}, Trans. Amer. Math. Soc. \textbf{164} (1972), 105--114. \MR{285589}

\bibitem{MR3076943}
D.~I. Moldavanski\u{\i}, \emph{On the residuality of {B}aumslag-{S}olitar groups (in {R}ussian)}, Chebyshevski\u{\i} Sb. \textbf{13} (2012), no.~1(41), 110--114. \MR{3076943}

\bibitem{MR0231897}
V.~P. Platonov, \emph{A certain problem for finitely generated groups (in {R}ussian)}, Dokl. Akad. Nauk BSSR \textbf{12} (1968), 492--494. \MR{231897}

\bibitem{MR0472792}
A.~A. Suslin, \emph{The structure of the special linear group over rings of polynomials (in {R}ussian)}, Izv. Akad. Nauk SSSR Ser. Mat. \textbf{41} (1977), no.~2, 235--252, 477. \MR{472792}

\bibitem{tholozan2022residually}
Nicolas Tholozan and Konstantinos Tsouvalas, \emph{Residually finite non linear hyperbolic groups}, 2022. arXiv:2207.14356.

\bibitem{MR0367080}
B.~A.~F. Wehrfritz, \emph{Generalized free products of linear groups}, Proc. London Math. Soc. (3) \textbf{27} (1973), 402--424. \MR{367080}

\bibitem{MR0335656}
\bysame, \emph{Infinite linear groups. {A}n account of the group-theoretic properties of infinite groups of matrices}, Ergebnisse der Mathematik und ihrer Grenzgebiete [Results in Mathematics and Related Areas], vol. Band 76, Springer-Verlag, New York-Heidelberg, 1973. \MR{335656}

\end{thebibliography}

\end{spacing}

$ $

{\small{}Donsung Lee; \href{mailto:disturin@snu.ac.kr}{disturin@snu.ac.kr}}{\small\par}

{\small{}Department of Mathematical Sciences and Research Institute
of Mathematics,}{\small\par}

{\small{}Seoul National University, Gwanak-ro 1, Gwankak-gu, Seoul,
South Korea 08826}{\small\par}

\clearpage{}

\pagebreak{}

\pagenumbering{arabic}

\renewcommand{\thefootnote}{A\arabic{footnote}}
\renewcommand{\thepage}{A\arabic{page}}
\renewcommand{\thetable}{A\arabic{table}}
\renewcommand{\thefigure}{A\arabic{figure}}

\setcounter{footnote}{0} 
\setcounter{section}{0}
\setcounter{table}{0}
\setcounter{figure}{0}
\end{document}